\documentclass[12pt]{article}

\usepackage{amsmath}
\usepackage{amssymb,amsthm}
\usepackage[arrow,matrix]{xy}
\usepackage{xcolor}

\newcommand{\af}{\mathfrak a}
\newcommand{\g}{\mathfrak g}
\newcommand{\A}{\mathbb A}
\newcommand{\C}{\mathbb C}
\newcommand{\R}{\mathbb R}
\newcommand{\bs}{\backslash}
\newcommand{\G}{\mathbf G}
\newcommand{\bK}{\mathbf K}
\newcommand{\bL}{\mathbf L}
\newcommand{\bN}{\mathbf N}
\newcommand{\bP}{\mathbf P}
\newcommand{\U}{\mathbf U}
\newcommand{\V}{\mathbf V}

\newcommand{\Ind}{\operatorname{Ind}}
\newcommand{\Infl}{\operatorname{Infl}}
\newcommand{\infl}{^\mathrm{infl}}
\newcommand{\can}{^\mathrm{can}}
\newcommand{\prin}{^\mathrm{prin}}
\newcommand{\Ad}{\operatorname{Ad}}
\renewcommand{\Re}{\operatorname{Re}}
\renewcommand{\Im}{\operatorname{Im}}
\newcommand{\Ker}{\operatorname{Ker}}
\newcommand{\Hom}{\operatorname{Hom}}
\newcommand{\GL}{\operatorname{GL}}
\newcommand{\Sp}{\operatorname{Sp}}
\newcommand{\tr}{\operatorname{tr}}
\newcommand{\id}{\mathrm{id}}

\newtheorem{theorem}{Theorem}
\newtheorem{lemma}{Lemma}
\newtheorem{cor}{Corollary}
\newtheorem{conjecture}{Conjecture}
\newtheorem{hypoth}{Hypothesis}

\begin{document}

\title{The trace formula and prehomogeneous vector spaces}

\author{Werner Hoffmann\thanks{This work was supported by the SFB~701 of the German Research Foundation.}}

\maketitle

\begin{abstract}We describe an approach to express the geometric side of the Arthur-Selberg trace formula in terms of zeta integrals attached to prehomogeneous vector spaces. This will provide explicit formulas for weighted orbital integrals and for the coefficients by which they are multiplied in the trace formula. We implement this programme for the principal unipotent conjugacy class. The method relies on certain convergence results and uses the notions of induced conjugacy classes and canonical parabolic subgroups. So far, it works for certain types of conjugacy classes, which covers all classes appearing in classical groups of absolute rank up to two.
\medskip

\noindent MSC-class:  11F72 (Primary) 11S90, 11M41 (Secondary)
\end{abstract}

\section*{Introduction}

The trace formula is an equality between two expansions of a certain distribution on an adelic group. The spectral side of the formula encodes valuable information about automorphic representations of the group. Although the geometric side is regarded to be the source of information, it is far from explicit. It is a sum of so-called weighted orbital integrals, each multiplied with a coefficient that carries global arithmetic information. So far, those coefficients have only been evaluated in some special cases. Arthur remarked on p.~112 of~\cite{A-int} that ``it would be very interesting to understand them better in other examples, although this does not seem to be necessary for presently conceived applications of the trace formula''. In the meantime, as reflected in the present proceedings, further applications have emerged which revive the interest in more detailed information on those coefficients and the weight factors of weighted orbital integrals.

The problem stems from the fact that the trace distribution is defined by an integral that does not converge without regularisation. The most successful method to accomplish this is Arthur's truncation~\cite{A-trI}. However, it does not yield useful formulas for the contributions from non-semisimple conjugacy classes to the geometric side. In the original rank-one trace formula (e.~g.,~\cite{A-rk1}), they were regularised by damping factors, which led to an expression containing zeta integrals. Shintani~\cite{Sh} observed that such integrals would also appear in the dimension formula for Siegel modular forms, which can be regarded as a special case of the trace formula, if one were able to prove convergence. The same method was applied by Flicker~\cite{Fl} to the group~$\GL_3$, but for groups of higher rank, the difficulties piled up. Arthur bypassed them by a clever invariance argument, which worked for unipotent conjugacy classes, and by reducing the general case to the unipotent one~\cite{A-mix}. The price to pay was that most coefficients and weight factors remained undetermined.

We take up the original approach and remove some of the obstacles on the way to express the regularised terms on the geometric side by zeta integrals. In many cases, these integrals are supported on prehomogeneous vector spaces which appear as subquotients of canonical parabolic subgroups of unipotent elements. Moreover, just as induced representations play an important role on the spectral side, we systematically apply the notion of induced conjugacy classes on the geometric side. So far, this approach has been successful for certain types of conjugacy classes, which suffice for a complete treatment of classical groups of absolute rank up to~$2$. The details, including the necessary estimates, can be found in a joint paper~\cite{HoWa} with Wakatsuki.

Over several years of work on this project, something like a general formula was gradually emerging, changing shape as more and more conjugacy classes with new features were covered. Incomplete as the results may be, they should perhaps be made available to a wider audience now together with an indication of the remaining difficulties.

Let us describe the setting in more detail. We consider a connected reductive linear algebraic group $G$ defined over a number field~$F$. The group $G(\A)$ of points with coordinates in the ring $\A$ of adeles of~$F$ acts by right translations on the homogeneous space $G(F)\bs G(\A)$, which carries an invariant measure coming from a Haar measure on $G(\A)$ and the counting measure on~$G(F)$. The resulting unitary representation $R_G$ of $G(\A)$ on the Hilbert space $L^2(G(F)\bs G(\A))$ can be integrated to a representation of the Banach algebra $L^1(G(\A))$, and for an element $f$ of the latter, $R_G(f)$ is an integral operator with kernel
\[
K_G(x,y)=\sum_{\gamma\in G(F)}f(x^{-1}\gamma y).
\]
If $G$ is $F$-anisotropic, then $G(F)\bs G(\A)$ is compact, so the integral
\[
J(f)=\int_{G(F)\bs G(\A)}K_G(x,x)\,d x
\]
converges for smooth compactly supported functions~$f$ and defines a distribution~$J$ on~$G(\A)$. Now we have the geometric expansion
\[
J(f)=\sum_{[\gamma]}\int_{G^\gamma(F)\bs  G(\A) }f(x^{-1}\gamma x)\,d x,
\]
where $G^\gamma$ is the centraliser of~$\gamma$, and the spectral expansion
\[
\tr R_G(f)=\sum_\pi a^G(\pi)\tr\pi(f),
\]
where $a^G(\pi)$ is the multiplicity of the irreducible representation $\pi$ of $ G(\A)$ in $L^2(G(F)\bs G(\A))$. The Selberg trace formula in this case is the identity
\[
\tr R_G(f)=J(f).
\]

If the centre of $G(\A)$ is non-compact, then $R_G(f)$ has no discrete spectrum, hence its trace is not defined. Either one has to fix a central character or one has to replace the group by its largest closed normal subgroup $G(\A)^1$ with compact centre. If $G$ has proper parabolic subgroups $P$ defined over~$F$, both sides of the formula will still diverge. One has to take into account the analogous unitary representations $R_P$ of $G(\A)^1$ on the spaces $L^2(N(\A)P(F)\bs G(\A)^1)$, where the letter $N$ will always denote the unipotent radical of the group $P$ in the current context. By choosing a Levi component $M$ of~$P$, one can view $R_P$ as the representation induced from the representation $R_M$, after the latter has been inflated to a representation of $P(\A)$ by composing it with the projection $P(\A)\to M(\A)$. The kernel function for $R_P(f)$ with $f\in C_c^\infty(G(\A)^1)$ is
\[
K_P(x,y)=\sum_{\gamma\in P(F)/N(F)}\int_{N(\A)}f(x^{-1}\gamma ny)\,dn,
\]
where we normalise the Haar measure on the group~$N(\A)$ in such a way that $N(F)\bs N(\A)$ has measure~1. This can be written as a single integral over~$P(F)N(\A)$, whose integrand is compactly supported locally uniformly in $x$ and~$y$. The trace distribution is defined as
\[
J^T(f)=\int_{G(F)\bs G(\A)^1}\sum_P K_P(x,x)\hat\tau_P^T(x)\,dx,
\]
where $P$ runs over all parabolic $F$-subgroups including $G$ itself. The functions $\hat\tau_P^T$ are, up to sign, certain characteristic functions on $G(\A)$ depending on a truncation parameter~$T$ and on the choice of a maximal compact subgroup $\bK$ of~$G(\A)$. We will recall their definition in section~\ref{trunc} below, noting for the moment that $\tau_G^T(x)=1$. Their alternating signs are responsible for cancellations that make the integrand rapidly decreasing and allowed Arthur to prove absolute convergence~\cite{A-trI}.

Actually, his argument was more subtle and led to a geometric expansion of~$J^T(f)$, later called the coarse geometric expansion. It represents an intermediate stage on the way to the fine geometric expansion~\cite{A-mix}. The latter depends on the choice of a finite set $S$ of valuations of $F$ including the Archimedean ones and has the shape
\[
J^T(f)=\sum_{[M]}\sum_{[\gamma]_{M,S}}a^M(S,\gamma)J_M^T(\gamma,f).
\]
Here $f$ is a smooth compactly supported function on~$G(F_S)^1$ suitably extended to~$G(\A)^1$, where $F_S$ is the product of completions $F_v$ of~$F$ with respect to $v\in S$. The summation runs over the conjugacy classes of Levi $F$-subgroups $M$ of~$G$ and, for each such class, over the classes of elements $\gamma$ with respect to the finest equivalence relation with the following properties. Elements with $M(F)$-conjugate semisimple components are equivalent, and elements with the same semisimple component $\sigma$ and $M_\sigma(F_S)$-conjugate unipotent components are also equivalent. The weighted orbital integral $J_M^T(\gamma,f)$ is an integral with respect to a certain non-invariant measure that is undetermined in general. It is supported on the $F_S$-valued points of the conjugacy class of $G$ induced from that of $\gamma$ in~$M$. The coefficients $a^M(S,\gamma)$ do not depend on the ambient group~$G$. They have been determined for semisimple elements~\cite{A-mix}, for $M$ of $F$-rank one~\cite{Ho0}, for $M=\GL_3$ \cite{Fl}, \cite{Ma} and, with the methods presented here, for the symplectic group of rank~two~\cite{HoWa}.

\section{Prerequisites}

In this section we collect some results in order to avoid interruptions of the arguments to follow. Unless stated otherwise, all affine varieties and linear algebraic groups that appear are assumed to be connected and defined over a given field~$F$. When we speak of orbits in a $G$-variety $V$ defined over~$F$, we mean geometric orbits defined over~$F$, i.~e., minimal $F$-subvarieties $O$ such that $O(F)$ is non-empty. This applies, in particular, to conjugacy classes. By Proposition~12.1.2 of~\cite{Sp}, every element of $V(F)$ belongs to an orbit, and an orbit remains a single orbit under base change to an extension field.

\subsection{Induction of conjugacy classes}

The following well-known result has been proved by Lusztig and Spaltenstein~\cite{LS} for unipotent conjugacy classes, and its extension to general conjugacy classes can be found in~\cite{Ho}.
\begin{theorem}\label{induction}
Let $P$ be a parabolic subgroup of a reductive group~$G$ with unipotent radical~$N$  and $C$ a conjugacy class in a Levi component~$M$ of~$P$.
Then there is a unique dense $P$-conjugacy class $C'$ in $CN$ and a unique conjugacy class $\tilde C$ in~$G$ such that $\tilde C\cap P=C'$.
\end{theorem}

We will write $\tilde C=\Ind_P^GC$ and $C'=\Infl_M^PC$. The map $\Ind_P^G$ is called induction of conjugacy classes from $M$ to $G$ via~$P$, and the map $\Infl_M^P$ will be called inflation of conjugacy classes from $M$ to~$P$. The Levi components of $P$ are naturally isomorphic to~$P/N$ and will be called Levi subgroups of~$G$. We denote by $P\infl$ the set of all elements $\gamma\in P$ for which the range of the endomorphism $\Ad\gamma-\id$ of the Lie algebra of $P$ contains the Lie algebra of~$N$.
\begin{theorem}\label{induction2}
\begin{enumerate}
\item[(i)] If $M$ is a Levi component of two parabolic subgroups $P$ and $Q$ of~$G$, then $\Ind_P^GC=\Ind_Q^GC$, whence this set can be denoted by $\Ind_M^G C$.
\item[(ii)] If $M\subset M'$ are Levi subgroups of~$G$, then $\Ind_M^G C=\Ind_{M'}^G\Ind_M^{M'} C$.
\item[(iii)] The union of all the sets $C'(F)$ with $C'=\Infl_M^PC$ for conjugacy classes $C$ in $M$ over~$F$ equals $P\infl(F)$.
\item[(iv)] Given~$\gamma\in G(F)$, the set $\mathcal P_\gamma\infl$ of parabolic subgroups $P$ such that $\gamma\in P\infl$ is a finite algebraic subset of the flag variety defined over~$F$.
\end{enumerate}
\end{theorem}
The first two assertions have been proved in~\cite{LS}, the other ones in~\cite{Ho}.

\subsection{Prehomogeneous varieties}
\label{pv}

Let $G$ be a linear algebraic group. A prehomogeneous $G$-variety is an irreducible $G$-variety $V$ possessing a dense $G$-orbit~$O$. The ``generic'' stabilisers $G^\xi$ (which may be non-connected) of elements $\xi\in O$ are then conjugate in~$G$. A nonzero rational function $p$ on $V$ is relatively $G$-invariant if there exists a character $\chi$ of $G$ such that, for all $g\in G$ and $x\in V$,
\[
p(gx)=\chi(g)p(x).
\]
A prehomogeneous $G$-variety $V$ is called special if every relative invariant (defined over any extension field of~$F$) is constant. This is the case if and only if the restriction homomorphism from the group $X(G)$ of algebraic characters of $G$ to $X(G^\xi)$ is an isomorphism.

\begin{theorem}\label{prehom}
Let $P$ be a parabolic subgroup of the reductive group~$G$ with unipotent radical~$N$ and let $N'\subset N''$ be normal unipotent subgroups of $P$.
\begin{enumerate}
\item[(i)] For any $\gamma\in P\infl$, the affine space $\gamma N''/N'$ is prehomogeneous under the action of the trivial connected component $P_{\gamma N''}$ of the stabiliser of $\gamma N''$ in~$P$ by conjugation.
\item[(ii)] If $C'$ is the $P$-conjugacy class of~$\gamma$, then the generic orbit is the projection of $C'\cap \gamma N''$, viz. $(C'\cap \gamma N'')N'/N'$.
\item[(iii)] The prehomogeneous variety $\gamma N/N'$ is special if and only if $\gamma N/N''$ and $\gamma N''/N'$ are special.
\end{enumerate}
\end{theorem}
This follows from Proposition~5 of~\cite{Ho}. Note that the action on the affine spaces in question is not always given by affine transformations.

A prehomogeneous $G$-variety is called a prehomogeneous vector space if it is a vector space and the action of $G$ is linear. A prehomogeneous vector space is called regular if the dual space $V^*$ is prehomogeneous for the contragredient action and the map $dp/p:O\to V^*$ is a dominant morphism for some relative invariant~$p$. The notion of $F$-regularity is defined in the obvious way.

Prehomogeneous vector spaces that are regular over a number field $F$ have been intensively studied because they give rise to zeta integrals \[
Z(\varphi,s_1,\dots,s_n)=\int_{G(\A)/G(F)}
|\chi_1(g)|^{s_1}\cdots|\chi_n(g)|^{s_n}
\sum_{\xi\in O(F)}\varphi(g\xi)\,dg,
\]
where $\varphi$ is a Schwarz-Bruhat function on~$V(\A)$ and the characters $\chi_i$ correspond to relative invariants $p_i$ which extend to regular functions on~$V$ and form a basis of the group of all relative invariants defined over~$F$. Here we preclude that the connected generic stabilisers $G_\xi$ have nontrivial $F$-rational characters, as the integral is otherwise divergent. (We will encounter prehomogeneous vector spaces, of incomplete type in the terminology of~\cite{Yu}, where this happens and one has to truncate the integrand.) A typical result of the classical theory is the following.
\begin{theorem}\label{zetaInt}
Suppose in addition that $G$ is $F$-anisotropic modulo centre. Let  $V=\bigoplus_{i=1}^nV_i$ be the splitting obtained by diagonalisation of the largest $F$-split torus in the centre of $G$ and choose $p_i$ depending only on the $i$-th component.
\begin{enumerate}
\item[(i)] The zeta integral converges absolutely when $\Re s_i>r_i$ for all~$i$, where $r_i=\dim V_i/\deg p_i$, and extends to a meromorphic function on~$\C^n$. Its only singularities are at most simple poles along the hyperplanes $s_i=r_i$ and $s_i=0$.
\item[(ii)] For each splitting of the index set $\{1,\dots,n\}$ into a disjoint union $I'\cup I''$ and the corresponding splitting $V=V'\oplus V''$, we have
\[
\lim_{s'\to r'}Z(\varphi,s',s'')\prod_{i\in I'}(s_i-r_i)
=Z''(\varphi'',s''),
\]
where $Z''$ is the zeta integral over
\[
\{g\in G(\A)\,\big|\,|\chi_i(g)|=1\,\forall\,i\in I'\}/G(F)
\]
of the function
\[
\varphi''(x'')=\int_{V'}\varphi(x',x'')\,dx'.
\]
\item[(iii)] For each splitting as above, we have the functional equation
\[
Z(\varphi,s',s'')=Z(\mathcal F'\varphi,r'-s',s''),
\]
where $\mathcal F'$ denotes the partial Fourier transform with respect to~$V'$.
\end{enumerate}
\end{theorem}
The convergence for large $\Re s_i$ has been proved in a rather general situation by Saito~\cite{Sai}. The present situation is much easier, since $G(\A)^1/G(F)$ is compact and the centre acts by componentwise multiplication. The proof of the remaining assertions goes hand in hand and proceeds as in~\cite{Sa}.

If we fix a finite set $S$ of places of $F$ containing the archimedean ones, and a lattice in $V(\A^S)$ with respect to the maximal compact subring of~$\A^S$, then every Schwartz-Bruhat function $\varphi$ on~$V(F_S)$ can be canonically extended to~$V(\A)$. For such functions, one obtains a decomposition
\[
Z(\varphi,s)=\sum_{[\xi]_S}\zeta(\xi,s)
\int_{G(F_S)/G^\xi(F_S)}
|\chi_1(g)|^{s_1}\cdots|\chi_n(g)|^{s_n}
\varphi(g\xi)\,dg
\]
over the finitely many $G(F_S)$-orbits $[\gamma]_S$ in $O(F_S)$, where the zeta functions $\zeta(\xi,s)$ encode valuable arithmetic information (see \cite{Ki} for the case $F=\mathbb Q$, $S=\{\infty\}$). We will not go into details here but rather describe a similar procedure for conjugacy classes in subsection~\ref{dampsection}.

\subsection{Canonical parabolic subgroups}

From now on, we assume that $G$ is reductive and $F$ has characteristic zero. Then every unipotent element of $G$ is of the form $\exp X$ for a nilpotent element of the Lie algebra $\g$ of~$G$. By the Jacobson-Morozov theorem (see \S3.3 of~\cite{CG} or \S11.2 of~\cite{Bou}), there is a homomorphism $\mathfrak{sl}_2\to\g$ such that $X$ is the image of~$\left(\begin{smallmatrix}0&1\\0&0\end{smallmatrix}\right)$. Let $H$ be the image of $\left(\begin{smallmatrix}1&0\\0&-1\end{smallmatrix}\right)$ and set $\mathfrak g_n=\{Z\in\mathfrak g\mid [H,Z]=nZ\}$, so that $X\in\g_2$. We consider the subalgebras
\[
\mathfrak q=\bigoplus_{n\ge0}\mathfrak g_n,\qquad
\mathfrak u=\bigoplus_{n>0}\mathfrak g_n,\qquad
\mathfrak u'=\bigoplus_{n>1}\mathfrak g_n,\qquad
\mathfrak u''=\bigoplus_{n>2}\mathfrak g_n
\]
and the subgroups
\[
Q=\operatorname{Norm}_G\mathfrak q,\qquad
U=\exp\mathfrak u,\qquad
U'=\exp\mathfrak u',\qquad
U''=\exp\mathfrak u''.
\]
It is well known that $\mathfrak q$ is a parabolic subalgebra with ideals $\mathfrak u$, $\mathfrak u'$ and~$\mathfrak u''$, where $[\mathfrak u,\mathfrak u]=\mathfrak u'$ and $[\mathfrak u,\mathfrak u']=\mathfrak u''$, and that $Q$ is a parabolic subgroup of $G$ with unipotent radical $U$ and normal subgroups $U$ and~$U'$. By results of Kostant and Mal'cev (see \cite{CG}, ch.~3), all of them are independent of the choice of the homomorphism $\mathfrak{sl}_2\to\g$ used in the definition. Moreover, $L=\operatorname{Cent}_GH$ is a Levi component of~$Q$. One calls $\mathfrak q$ the canonical parabolic subalgebra of~$X$. If $\exp X$ is the unipotent component in the Jordan decomposition of an element $\gamma\in G$, we call $Q$ the canonical parabolic subgroup of~$\gamma$. Moreover, we denote by $Q\can$ the set of elements of $G$ whose canonical parabolic is~$Q$.
\begin{theorem}\label{canpar}
\begin{enumerate}
\item[(i)] If $X\in\g(F)$ resp.\ $\gamma\in G(F)$, then $\mathfrak q$ resp.\ $Q$ is defined over~$F$. We can choose $H\in\g(F)$, and then $L$ is defined over~$F$.
\item[(ii)] The vector space $\mathfrak u'/\mathfrak u''$ with the adjoint action of~$L\cong Q/U$ is a regular prehomogeneous vector space. In the situation of~(i) it is  $F$-regular.
\item[(iii)] If $Q$ is the canonic parabolic and $C$ the conjugacy class of an element~$\gamma$, then $C\cap Q\can$ is the conjugacy class of $\gamma$ in~$Q$. If $\gamma$ is unipotent, then it is open and dense in~$U'$ and invariant under translations by elements of~$U''$.
\item[(iv)] If $\gamma$ is a unipotent element of a parabolic subgroup $P$ of~$G$, then $U\subset P$.
\end{enumerate}
\end{theorem}
Assertion~(i) is obvious. Proofs of assertion~(ii) (which is is folklore) and (iii) can be found in~\cite{Ho}. Theorem~2 of that paper also contains a version of the last statement for mixed elements, but that seems to be less useful for our purposes.

In order to prove~(iv), observe that $\gamma$ is contained in a Borel subgroup of~$P$, and that one can choose $H$ in the Lie algebra of a maximal torus $T$ in~$B$. Then $H$ lies in the closure of the positive chamber of $\mathfrak t$ with respect to~$B$ (see \S~3.5 of~\cite{CG}). Since $\mathfrak u$ is the sum of the root spaces for roots $\alpha$ with~$\alpha(H)>0$, it is contained in the unipotent radical of $B$ and hence in~$P$.

\subsection{Mean values}

A mean value formula has been proved by Siegel for the action of $\operatorname{SL}_n(\R)$ on $\R^n$ ($n>1$), generalised by Weil~\cite{We} to the adelic setting and by Ono~\cite{Ono} to the following general case.
\begin{theorem}
If $O$ is a special $G$-homogeneous variety over a number field~$F$ with trivial groups $\pi_1(O(\C))$, $\pi_2(O(\C))$ and $X(G)$, then
\[
\int_{G(\A)/G(F)}\sum_{\xi\in O(F)}h(g\xi)\,dg=\int_{G(\A)O(F)}h(x)\,dx
\]
for $h\in C_c^\infty(G(\A)O(F))$ and a suitable normalisation of invariant measures.
\end{theorem}
Actually, Ono imposed the additional assumption that $[G(\A)\xi\cap O(F):G(F)]$ be independent of~$\xi\in O(F)$, but this is automatically satisfied by Proposition~2.3 of~\cite{MoWa}. Moreover, he used the term ``special'' only under the assumption that the group $X(G)$ is trivial. With our wider definition, the theorem is still valid if we replace $G$ by its derived subgroup~$G'$, because the map $G'/G'^\xi\to G/G^\xi$ is an isomorphism.

If $O$ is the generic orbit in a special prehomogeneous affine space~$V$, then the first two homotopy groups are automatically trivial. In fact, for any Lipschitz map $\phi:S^i\to O(\C)$, the map $\psi:W(\C)\times S^i\times\R$ given by $\psi(w,s,t)=tw+(1-t)\phi(s)$ has range of Hausdorff dimension at most $2\dim W+i+1$. For $i\le 2$, this is less than $\dim_\R V(\C)$, so we can choose $x\in V(\C)$ not in those ranges and get a null-homotopy $\phi_t(s)=tx+(1-t)\phi(s)$ in~$O(\C)$.

We need a slightly different version of the above theorem.
\begin{theorem}\label{mean}
If $V$ is a torsor under a unipotent group~$N$ and the group $G$ with  trivial $X(G)$ acts on the pair $(N,V)$ by automorphisms, so that $V$ is a special prehomogeneous $G$-space with generic orbit $O$ and the orbit map $G\to O$ has local sections, then
\[
\int_{G(\A)/G(F)}\sum_{\xi\in O(F)}h(g\xi)\,dg
=\int_{G(\A)/G(F)}\int_{V(\A)}h(gx)\,dx\,dg
\]
for $h\in C_c^\infty(V(\A))$, provided we normalise the measure on $V(\A)$ so that $V(\A)/N(F)$ has measure~$1$.
\end{theorem}
\begin{proof}
Since the orbit map $g\mapsto gv$ for $v\in O(F)$ has local sections, which are defined over~$F$ according to our standing assumption, it maps $G(F)$ onto $O(F)$ and $G(\A)$ onto~$O(\A)$. In particular, Ono's additional condition is trivially satisfied. The complement of $O$ is a subvariety $W$ of codimension greater than one by Lemma~7 of~\cite{Ho}, hence a null set for the $N(\A)$-invariant measure on $V(\A)$. That measure is also $G(\A)$-invariant, hence its restriction to $O(\A)$ coincides with the measure in Ono's theorem, in which we may replace the domain of integration on the right-hand side by~$V(\A)$. We may also replace the integrand $h(x)$ by $h(g_1x)$, where $g_1\in G(\A)$ is arbitrary, and then integrate the right-hand side over~$g_1$, as the mesure of $G(F)\bs G(\A)$ is finite due to $X(G)=\{1\}$. This proves the claim up to the normalisation of measures and the extension to~$C_c^\infty(V(\A))$.

There is an alternative, though less elegant, proof, which provides these facts. One reduces the assertion to the case of abelian $N$ using a central series of a general unipotent group $N$ and Proposition~5 of~\cite{Ho}. In the abelian case one proceeds as in~\cite{We}.
\end{proof}

In the situation of Theorem~\ref{prehom}, $\gamma N/N'$ is an $N/N'$-torsor, on which $P_{\gamma N}$ acts by automorphisms. In order to apply Theorem~\ref{mean}, we need the following hypothesis about a parabolic subgroup $P$ of a reductive group $G$ with unipotent radical $N$ and a conjugacy class $C$ in~$P/N$:
\begin{hypoth}\label{NC}
There is a normal unipotent subgroup $N^C$ of $P$ such that, for $\gamma\in C'=\Infl^PC$,
\begin{enumerate}
\item[(i)] the prehomogeneous affine space $\gamma N/N^C$ is special under  $P_{\gamma N}$ and the generic orbit map has local sections,
\item[(ii)] all elements of $\gamma N^C\cap C'$ have the same canonical parabolic.
\end{enumerate}
\end{hypoth}
We call this a hypothesis rather than a conjecture because if it is not generally true, we may at least treat those conjugacy classes to which it applies. In fact, it has been checked for all classical groups up to rank~3.

As the notation suggests, there should be a canonical choice for~$N^C$. By Lemma~8 of~\cite{Ho}, there is a largest normal unipotent subgroup of $P$ with property~(ii), and under Hypothesis~\ref{NC} it will then also have property~(i) in view of Theorem~\ref{prehom}. In general, however, it seems not to be the correct choice for our purposes. We certainly assume, as we may, that $(\gamma N\gamma^{-1})^{\gamma C\gamma^{-1}}=\gamma N^C\gamma^{-1}$ for all $\gamma\in G(F)$.

\subsection{$(G,Q)$-families}
\label{GQfam}

In section~\ref{prin}, we will need an analogue of the notion of $(G,M)$-families (see section~17 of~\cite{A-int}) in which the Levi subgroup $M$ is replaced by a parabolic subgroup~$Q$. First we recall the pertinent notation.

For every connected linear algebraic group~$P$ defined over~$F$, we denote by $\af_P$ the real vector space of all homomorphisms from the group $X(P)_F$ of $F$-rational characters of $P$ to the group~$\R$. If $P=MN$ is a Levi decomposition and $A$ the largest split torus in the centre of~$M$, then the natural homomorphisms $\af_A\to\af_M\to\af_P$ are isomorphisms, and the set $\Delta_P$ of fundamental roots of $A$ in $\mathfrak n$ can be regarded as a subset of the dual space~$\af_P^*$ independent of the choice of~$A$. Moreover, if $Q\subset P$ are parabolic subgroups of a reductive group~$G$, we obtain natural maps $\af_P\rightleftarrows\af_Q$, which induce a splitting $\af_Q=\af_Q^P\oplus\af_P$. The coroots $\check\alpha$ are originally only defined for roots $\alpha$ of a maximal split torus, hence for the elements of $\Delta_Q$, when $Q$ is a minimal parabolic, but if $\beta=\alpha|_{\af_P}$ is nonzero, we may define $\check\beta$ as the projection of $\check\alpha$ to~$\af_P$. These coroots form a basis of~$\af_P^G$, and we denote the dual basis of $(\af_P/\af_G)^*$ by~$\hat\Delta_P$, whose elements $\varpi$ are called fundamental weights. The basis dual to~$\Delta_P$, whose elements are called fundamental coroots, is in bijection with $\Delta_P$ and hence with~$\hat\Delta_P$. Following~\cite{A-int}, the fundamental coroot corresponding to $\varpi\in\hat\Delta_P$ will be denoted by~$\check\varpi$.

The charactersitic functions of the chamber $\af_P^+=\{H\in\af_P\mid\alpha(H)>0\;\forall\,\alpha\in\Delta_P\}$ and the dual cone ${}^+\af_P=\{H\in\af_P\mid\varpi(H)>0\;\forall\,\varpi\in\hat\Delta_P\}$ are denoted by $\tau_P$ and $\hat\tau_P$, respectively. Their Fourier transforms
\[
\hat\theta_P(-\lambda)^{-1}
=\int_{(\af_P^G)^+}e^{\langle\lambda,H\rangle}\,dH,\qquad
\theta_P(-\lambda)^{-1}
=\int_{{}^+(\af_P^G)}e^{\langle\lambda,H\rangle}\,dH
\]
(mind the swap of the accent) are defined for complex-valued linear functions $\lambda$ on~$\af_P$ with positive real part on the support. An easy computation yields
\[
\hat\theta_P(\lambda)=\hat\eta_P\prod_{\varpi\in\hat\Delta_P}\langle\lambda,\check\varpi\rangle,\qquad
\theta_P(\lambda)=\eta_P\prod_{\alpha\in\Delta_P}
\langle\lambda,\check\alpha\rangle,
\]
where the constants $\hat\eta_P$ and $\eta_P$ depend on the Haar measure on~$\af_P^G$.

The parabolic subgroups $R$ of a Levi component $M$ of~$P$ are in bijection with the parabolic subgroups $Q$ of~$P$ via $R\mapsto RN$, $Q\mapsto M\cap Q$. Prompted by the equality $\af_R^M=\af_Q^P$, one indexes the objects associated to the pair $(M,R)$ in place of~$(G,Q)$ by the pair $(P,Q)$ of parabolics of~$G$, like $\Delta_Q^P$ etc. The upper index $G$ may sometimes be omitted, like for $\tau_P^G$ etc.

For parabolics $P\subset P'$ containing~$Q$, we denote the restriction of a linear function $\lambda$ on $\af_Q$ to the subspace $\af_P^{P'}$ by~$\lambda_P^{P'}$, where the upper index $G$ and the lower index $Q$ may be omitted. The relative versions of the above Fourier transforms are extended to all $\lambda$ by setting
\[
\hat\theta_P^{P'}(\lambda)=\hat\theta_P^{P'}(\lambda_P),\qquad
\theta_P^{P'}(\lambda)=\theta_P^{P'}(\lambda_P),
\]
and similar remarks apply to the functions $\tau_P^{P'}$ and $\hat\tau_P^{P'}$. We assume that the measures on all the spaces $\af_P^{P'}$ are normalised in a compatible way.

With notational matters out of the way, we now define a $(G,Q)$-family to be a family of holomorphic functions $c_P(\lambda)$ indexed by the parabolic subgroups $P$ containing~$Q$ and defined for $\Re\lambda$ in a neighbourhood of zero in~$(\af_Q)_\C^*$ such that, for any two parabolics $P\subset P'$ containing~$Q$,
\[
\lambda_P^{P'}=0\quad\Rightarrow\quad c_P(\lambda)=c_{P'}(\lambda).
\]
This condition does not get weaker if we require it only for $P$,  $P'$ with $\dim\af_P^{P'}=1$. We say that the $(G,Q)$-family is frugal (resp. cofrugal) if $c_P(\lambda)=c_Q(\lambda_P)$ (resp. $c_P(\lambda)=c_G(\lambda^P)$) for all~$P$, where $\lambda_P$ (resp.~$\lambda^P$) is extended to $\af_Q$ so that it vanishes on~$\af_Q^P$ (resp.\ on~$\af_P$). Every holomorphic function~$c_Q$ on~$\af_Q$ (resp. $c_G$ on~$\af_Q^G$) determines a frugal (resp.\ cofrugal) $(G,Q)$-family.
\begin{lemma}\label{GQ}
For each $(G,Q)$-family of functions~$c_P$, the meromorphic function
\[
c_Q'(\lambda)=\sum_{P\supset Q}\epsilon_Q^P c_P(\lambda)\hat\theta_Q^P(\lambda)^{-1}\theta_P(\lambda)^{-1},
\]
where $\epsilon_Q^P=(-1)^{\dim\af_Q^P}$, is holomorphic for $\Re\lambda$ in a neighbourhood of zero.
\end{lemma}
The special case for frugal families is Lemma~6.1 of~\cite{A-inv} (with the roles of $P$ and $Q$ interchanged). As in that source, we could also prove a version for smooth functions defined for purely imaginary $\lambda$ only, although it does not seem to have applications. One may compute the value $c_Q'(0)$ by setting $\lambda=z\lambda_0$ for any fixed $\lambda_0$ not on any singular hyperplane and applying l'Hospital's rule to the resulting function of $z\in\C$, reduced to a common denominator.
\begin{proof}
For each fundamental root $\alpha\in\Delta_Q$, we denote the corresponding fundamental weight by~$\varpi_\alpha$. Let $P$ be a parabolic subgroup containing~$Q$. For each $\alpha\in\Delta_Q\setminus\Delta_Q^P$, the projection of $\check\alpha$ along $\af_Q^P$ onto $\af_P^G$ is a fundamental coroot $\check\alpha_P$, and for each $\alpha\in\Delta_Q^P$, the projection of $\check\varpi_\alpha$ along $\af_P^G$ onto $\af_Q^P$ is a fundamental coweight $\check\varpi_\alpha^P$. All the fundamental coroots of $\af_P^G$ in~$P$ and fundamental coweights of $\af_Q^P$ in $Q/N$ arise in this way.

We fix a fundamental root $\beta\in\Delta_Q$ and denote by $Q'$ the parabolic with $\Delta_Q^{Q'}=\{\beta\}$. Then there is a unique bijection $P\mapsto P'$ from $\{P\supset Q\mid P\not\supset Q'\}$ onto $\{P\supset Q'\}$ such that $\Delta_Q^{P'}=\Delta_Q^P\cup\{\beta\}$. The elements $\check\beta_P$ and $\check\varpi_\beta^{P'}$ as well as the differences $\check\alpha_P-\check\alpha_{P'}$ for every $\alpha\in\Delta_Q\setminus\Delta_Q^{P'}$ and $\check\varpi_\alpha^P-\check\varpi_\alpha^{P'}$ for every $\alpha\in\Delta_Q^P$ lie in the one-dimensional subspace~$\af_P^{P'}$. Together with the defining property of the $(G,Q)$-family this implies that the difference of
\[
c_P(\lambda)
\prod_{\alpha\in\Delta_Q^P}\lambda(\check\varpi_\alpha^{P'})
\prod_{\alpha\in\Delta_Q\setminus\Delta_Q^{P'}}\lambda(\check\alpha_{P'})
\]
and
\[
c_{P'}(\lambda)
\prod_{\alpha\in\Delta_Q^P}\lambda(\check\varpi_\alpha^P)
\prod_{\alpha\in\Delta_Q\setminus\Delta_Q^{P'}}\lambda(\check\alpha_P)
\]
vanishes on the hyperplane defined by $\lambda|_{\af_P^{P'}}=0$ and is therefore a multiple of the proportional linear forms $\lambda(\check\beta_P)$ and~$\lambda(\check\varpi_\beta^{P'})$. Dividing by
\[
\hat\theta_Q^P(\lambda)\theta_P(\lambda)\cdot
\hat\theta_Q^{P'}(\lambda)\theta_{P'}(\lambda),
\]
and remembering its compatible normalisation, we see that
\[
c_P(\lambda)\hat\theta_Q^P(\lambda)\theta_P(\lambda)
-c_{P'}(\lambda)\hat\theta_Q^{P'}(\lambda)\theta_{P'}(\lambda)
\]
is singular at most for those $\lambda$ which vanish on some one-dimensional subspace $\af_R^{R'}$ with $\Delta_Q^{R'}\setminus\Delta_Q^R=\{\alpha\}$ for some $\alpha\ne\beta$. Multiplying by $\epsilon_Q^P$ and summing over~$P$, we see that the same is true of~$c_Q'(\lambda)$. Since $\beta$ was arbitrary, we are done.
\end{proof}
\begin{lemma}\label{G'}
Let $X\in\af_Q^G$.
\begin{enumerate}
\item[(i)] If $c_P=e^{\langle\lambda,X_P\rangle}$, then $c_Q'(\lambda)$ is the Fourier transform of the function
\[
\Gamma_Q'(H,X)=\sum_{P\supset Q}\epsilon_P\tau_Q^P(H)\hat\tau_P(H-X).
\]
\item[(ii)] If $c_P=e^{\langle\lambda,X^P\rangle}$, then $c_Q'(\lambda)$ is the Fourier transform of the function
\[
\Gamma_Q''(H,X)=\sum_{P\supset Q}\epsilon_P\tau_Q^P(H-X)\hat\tau_P(H).
\]
\item[(iii)] For $H$ outside a finite union of hyperplanes, we have
\[
\Gamma_Q''(H,X)=\epsilon_Q^G\Gamma_Q'(X-H,X).
\]
\end{enumerate}
\end{lemma}
\begin{proof}
Assertion (i) is Lemma~2.2 of~\cite{A-inv}, and the proof of assertion~(ii) is analogous. The substitution of $X-H$ for~$H$ has on the Fourier transform the effect of substituting $-\lambda$ for $\lambda$ and multiplying by~$e^{\langle\lambda,X\rangle}$. Since $\hat\theta_Q^P(\lambda)\theta_P(\lambda)$ is homogeneous of degree~$\dim\af_Q^G$ and $X=X_P+X^P$, the two sides of the asserted equality are characteristic functions of polyhedra with equal Fourier transforms.
\end{proof}

Given a $(G,Q)$-family and a parabolic $P\supset Q$, we obtain a $(G,P)$-family by restricting the functions $c_{P'}$ with $P'\supset P$ to the subspace~$\af_P$, and we obtain an $(M,M\cap Q)$-family, where $M$ is a Levi component of~$P$, by setting
\[
c^P_{M\cap P'}(\lambda)=c_{P'}(\lambda).
\]
Checking the condition for such families is straightforward.
\begin{lemma}\label{recQM}
\begin{enumerate}
\item[(i)] For frugal $(G,Q)$-families, the definition of $c_Q'$ is equivalent to the identity
\[
c_Q(\lambda)\theta_Q(\lambda)^{-1}=\sum_{P\supset Q}
c_P'(\lambda_P)\theta_Q^P(\lambda)^{-1}.
\]
\item[(ii)] For cofrugal $(G,Q)$-families, the definition of $c_Q'$ is equivalent to the identity
\[
c_G(\lambda)\hat\theta_Q(\lambda)^{-1}=\sum_{P\supset Q}\epsilon_Q^P
(c^P)'_{M\cap Q}(\lambda)\hat\theta_P(\lambda)^{-1}.
\]
\end{enumerate}
\end{lemma}
Note that both identities in (ii) can be read as recursive definitions by isolating the term with $P=Q$ or $P=G$, resp. See equation~(17.9) in~\cite{A-int} and equation~(6.2) in~\cite{A-inv} for the frugal case.
\begin{proof}
For each of the four required implications, one starts with the right-hand side of the equation to be proved and plugs in the hypothesis. Then one interchanges summations and uses the fact that the expressions
\[
\sum_{P\supset Q}\epsilon_P
\hat\theta_Q^P(\lambda)^{-1}\theta_P(\lambda)^{-1},\qquad
\sum_{P\supset Q}\epsilon_P
\theta_Q^P(\lambda)^{-1}\hat\theta_P(\lambda)^{-1}
\]
are $1$ for $Q=G$ and\/ $0$ otherwise, which follows from equations (8.10) and (8.11) of~\cite{A-int}.
\end{proof}
The elementwise product of two $(G,Q)$-families is again a $(G,Q)$-family.
\begin{lemma}\label{prodGQ}
Given two $(G,Q)$-families of functions~$c_P$ and~$d_P$, of which the former family is cofrugal or the latter family is frugal, we have the splitting formula
\[
(cd)_Q'(\lambda)=\sum_{P\supset Q}
(c^P)'_{M\cap Q}(\lambda)d_P'(\lambda_P).
\]
\end{lemma}
\begin{proof}
The proof for frugal $d$ is analogous to the case of $(G,L)$-families. Using the relative version of the first identity of Lemma~\ref{recQM}(ii) with $P'$ in place of~$Q$, we get
\begin{align*}
(cd)_Q'(\lambda)&=\sum_{P'\supset Q}\epsilon_Q^{P'}
c_{P'}(\lambda)\hat\theta_Q^{P'}(\lambda)^{-1}
\sum_{P\supset P'}d_P'(\lambda_P)\theta_{P'}^P(\lambda)^{-1}\\
&=\sum_{P\supset Q}d_P'(\lambda_P)
\sum_{P':Q\subset P'\subset P}\epsilon_Q^{P'}c_{P'}(\lambda)
\hat\theta_Q^{P'}(\lambda)^{-1}\theta_{P'}^P(\lambda)^{-1}.
\end{align*}
Similarly, using the second identity of Lemma~\ref{recQM}(ii) with $M'$ in place of~$G$, we get
\begin{align*}
(cd)_Q'(\lambda)&=\sum_{P'\supset Q}\epsilon_Q^{P'}
d_{P'}(\lambda)\theta_{P'}(\lambda)^{-1}
\sum_{P:Q\subset P\subset P'}\epsilon_Q^P
(c^P)_{M\cap Q}'(\lambda)\hat\theta_P^{P'}(\lambda)^{-1}\\
&=\sum_{P\supset Q}(c^P)_{M\cap Q}'(\lambda)
\sum_{P'\supset P}\epsilon_P^{P'}d_{P'}(\lambda)
\hat\theta_P^{P'}(\lambda)^{-1}\theta_{P'}(\lambda)^{-1}.
\end{align*}
Now it remains to apply the definitions of~$(c^P)_Q'$ and $d_P'$, resp. 
\end{proof}
 
\section{The geometric side of the trace formula}

\subsection{Truncation}
\label{trunc}

Before introducing a new sort of geometric expansion, let us supply the details omitted in the introduction. We fix a number field~$F$ and denote the product of its archimedean completions by~$F_\infty$. For any linear algebraic group $G$, tacitly assumed to be connected and defined over~$F$, the group of continuous homomorphisms from $G(\A)$ to the additive group~$\R$ has the structure of a real vector space. We denote its dual space by~$\af_G$ and define a continuous homomorphism
$H_G:G(\A)\to\af_G$ by $\chi(g)=\langle \chi,H_G(g)\rangle$ for all $\chi\in\af_G^*$. The kernel of $H_G$ is then the group~$G(\A)^1$. From now on, the letter $G$ will be reserved for a reductive group.

In order to save space, the set $V(F)$ of $F$-rational points of any affine $F$-variety~$V$ will henceforth simply be denoted by~$V$ and the set of its adelic points by the corresponding boldface letter~$\mathbf V$. As an exception, the letter $\bK$ will denote a maximal compact subgroup of our adelic group~$\G$ such that $G(F_\infty)\bK$ is open and $\G=\bP\bK$ for every parabolic subgroup~$P$. One extends the map $H_P:\bP\to\af_P$ to $\G$ by seting $H_P(pk)=H_P(p)$ for $p\in\bP$ and $k\in\bK$, not indicating the dependence on the choice of $\bK$ in the notation.

A truncation parameter $T$ for the pair $(\G,\bK)$ is a family of elements $T_P\in\af_P$ indexed by the parabolic subgroups such that the modified maps $H_P^T(x)=H_P(x)-T_P$ satisfy $H_{\gamma P\gamma^{-1}}^T(\gamma x)=H_P^T(x)$ for all $\gamma\in G$ and such that $H_{P'}^T(x)$ is the projection of $H_P^T(x)$ for arbitrary parabolics $P\subset P'$. Thereby we have eliminated the need for standard parabolic subgroups. The set of truncation parameters has the structure of an affine space such that the evaluation at any minimal parabolic~$P_0$ is an isomorphism onto~$\af_{P_0}$.

For any maximal parabolic~$P'$, let $-\hat\tau_{P'}^T$ be characteristic function of
\[
\{x\in \G \mid H_{P'}^T(x)\in\af_{P'}^+\},
\]
where $\af_{P'}^+$ denotes the positive chamber in~$\af_{P'}$. For general~$P$, set
\[
\hat\tau_P^T(x)=\prod_{\substack{P'\supset P\\\text{max.}}}\hat\tau_{P'}^T(x).
\]
Thereby the usual sign factors $\epsilon_P=(-1)^{\dim\af_P^G}$ in the integrand of $J^T(f)$ have been incorporated into these cut-off functions. We mention that the integral converges for all values of~$T$ (see~\cite{Ho1}) and depends polynomially on~$T$ (see~\cite{A-inv}).

\subsection{Expansion in terms of geometric conjugacy classes}

In the distribution $J^T(f)$, one cannot isolate the contribution of a group-theoretic conjugacy class in $G(F)$ because the representatives of a coset $\gamma N$ appearing in $K_P$ belong to various conjugacy classes. In the coarse geometric expansion~(see~\cite{A-trI}), conjugacy has therefore been replaced by a coarser equivalence relation for which all elements in such cosets are equivalent. The finest such relation turns out to be just conjugacy of semisimple components. The fine geometric expansion is based on an intermediate refinement that depends on a choice of a finite set of places of~$F$, but it is still not fully explicit. We propose to use geometric conjugacy for a start, deferring the finer expansions to the later step of stabilisation. It is induction of conjugacy classes that makes this work.

Let $P$ be a parabolic subgroup of~$G$ with unipotent radical~$N$. Recall that induction as defined in Theorem~\ref{induction} is actually a map from the set of conjugacy classes in~$P/N$ to conjugacy classes in~$G$ that does not depend on the choice of a Levi component. We define the contribution of a geometric conjugacy class $C$ in $G$ to the kernel function $K_P$ as
\[
K_{P,C}(x,y)=\sum_{\substack{D\subset P/N\\\Ind_P^GD=C}}
\sum_{\gamma\in D}\int_\bN f(x^{-1}\gamma ny)\,dn,
\]
where we denote conjugacy classes in $P/N$ by~$D$, the letter $C$ being reserved for conjugacy classes in~$G$. We stick to tradition and avoid the awkward expression $\gamma N\in D$ under the summation sign. Alternatively, we can write
\[
K_{P,C}(x,y)=\sum_{\gamma\in(C\cap P)N/N} \int_\bN f(x^{-1}\gamma ny)\,dn.
\]
The convergence of $K_P$ implies that of its subsum $K_{P,C}$, and it is obvious that
\[
K_P=\sum_C K_{P,C}.
\]
The contribution of the class $C$ to the trace distribution is defined formally as
\[
J_C^T(f)=\int_{G\bs\G^1}\sum_P K_{P,C}(x,x)\hat\tau_P^T(x)\,dx,
\]
but its convergence and the validity of the expansion
\[
J^T(f)=\sum_CJ_C^T(f)
\]
depends on the following condition.
\begin{conjecture}\label{expansion}
For $f\in C_c^\infty(G(\A)^1)$, we have
\[
\sum_C\int_{G\bs\G^1}\left|\sum_P K_{P,C}(x,x)\hat\tau_P^T(x)\right|dx<\infty.
\]
\end{conjecture}
This statement would also make sense for function fields~$F$. It is a version of the convergence theorem of~\cite{A-trI}, but the proof will require more delicate estimates, not to mention its extension to rapidly decreasing smooth integrable functions $f$ of noncompact support. The merit of such a result would depend on our ability to find a useful alternative description of the distributions $J_C^T(f)$. We will see that on the way to this goal even more subtle convergence results are needed.

\section{Rearranging the geometric side}

\subsection{Replacing integrals by sums}

In the term of~$K_{P,C}$ corresponding to a conjugacy class~$D$ in~$P/N$, the integral over $\bN$ can be split into an integral over $\bN/\bN^D$ and an integral over~$\bN^D$ in the notation of Hypothesis~\ref{NC}. We want to replace the first of these integrals by a sum. This is analogous to Theorem~8.1 in~\cite{A-trI}.

Note that $D$ defines a $P$-conjugacy class $D'=\Infl^PD$ and, conversely, each $P$-conjugacy class $D'$ in $C\cap P$ determines a conjugacy class $D=D'N/N$ in~$P/N$. We define the modified kernel function
\[
\tilde K_{P,C}(x,y)=
\sum_{\substack{D\subset P/N\\\Ind_P^GD=C}}\sum_{\gamma\in D'N^D/N^D}\int_{\bN^D}
f(x^{-1}\gamma n'x)\,dn'
\]
and, formally, the modified distribution
\[
\tilde J_C^T(f)=\int_{G\bs\G^1}
\tilde K_{P,C}(x,y)\hat\tau_P^T(x)\,dx.
\]
\begin{hypoth}\label{hypoth2}
\begin{enumerate}
\item[(i)] The analoge of Conjucture~\ref{expansion} is true for~$\tilde J_C^T(f)$.
\item[(ii)] For all parabolic subgroups~$P$ and conjugacy classes $D\subset P/N$, we have
\begin{multline*}
\int_{P\bs\G^1}\sum_{\delta\in D}\Biggl|
\sum_{\gamma\in(C\cap\delta N)N^D/N^D}
\int_{\bN^D}f(x^{-1}\gamma n'x)\,dn'\\
-\int_\bN f(x^{-1}\delta nx)\,dn
\Biggr|\hat\tau_P^T(x)\,dx<\infty.
\end{multline*}
\end{enumerate}
\end{hypoth}
This hypothesis is automatic for groups of $F$-rank one, as $N^D=N$ in that case, and has been checked for classical groups of absolute rank~2 in~\cite{HoWa}.
\begin{lemma}\label{int2sum}
Under Conjecture~\ref{expansion} and Hypotheses~\ref{NC} and~\ref{hypoth2}, we have
\[
J_C^T(f)=\tilde J_C^T(f).
\]

\end{lemma}

\begin{proof}
Granting Hypothesis~\ref{NC}, Theorem~\ref{mean} yields the vanishing of
\[
\int_{P_{\delta N}'\bs\bP_{\delta N}'}
\left(\sum_{\gamma\in(C\cap\delta N)N^D/N^D}
h(p_1^{-1}\gamma p_1)
-\int_{\bN/\bN^D} h(p_1^{-1}\delta n''p_1)\,dn''\right)dp_1
\]
for all $\delta\in D$ and $h\in C_c^\infty(\delta\bN/\bN^D)$, where $P'_{\delta N}$ denotes the derived group of~$P_{\delta N}$. We plug in
\[
h(v)=\hat\tau_P^T(y)\int_{\bN^D}f(x^{-1}p_2^{-1}v np_2x)\,dn
\]
with $p_2\in\bP^1$ and $y\in\G^1$, and substitute $p_1n=n'p_1$ in the integral over~$\bN^D$. Observing that the domain of integration over~$p_1$ can be written as $P\bs P\bP_{\delta N}'$, we integrate over $p_2\in P\bP_{\delta N}'\bs \bP^1$. Then we sum over $\delta\in D$ and take the combined integral over $p=p_1p_2\in P\bs\bP^1$ outside the sum. Combining the summation over $\gamma$ with that over $\delta$ into one summation and combining the integral over $n''$ with that over~$n'$ into one integral, we obtain
\begin{multline*}
\hat\tau_P^T(y)\int_{P\bs\bP^1}
\Biggl(\sum_{\gamma\in(C\cap DN)N^D/N^D}
\int_{\bN^D}f(y^{-1}p^{-1}\gamma n'py)\,dn'\\
-\sum_{\delta\in D}\int_{\bN} f(y^{-1}p^{-1}\delta n py)
\,dn\Biggr)\,dp=0.
\end{multline*}
Then we integrate over $y\in\bP^1\bs\G^1$ and combine this integral with that over $p$, observing that $\hat\tau_P^T(y)=\hat\tau_P^T(py)$, to get
\begin{multline*}
\int_{P\bs\G^1}
\Biggl(\sum_{\gamma\in D'N^D/N^D}
\int_{\bN^D}f(x^{-1}\gamma n'x)\,dn'\\
-\sum_{\delta\in D}\int_{\bN} f(x^{-1}\delta nx)
\,dn\Biggr)\hat\tau_P^T(x)\,dx=0.
\end{multline*}
All of these operations are justified under Hypothesis~\ref{hypoth2}(i).

We sum this expression over the finitely many standard parabolics $P$ and respective classes~$D$. Then we split the integral into an integral over $G\bs\G^1$ and a sum over $P\bs G$ and interchange the latter integral with the former sum. The latter sum can be replaced by a sum over all parabolic subgroups conjugate to~$P$, because the relevant objects attached to different parabolics (the unipotent radical~$N$, the set of conjugacy classes $D$ in $P/N$ with $\Ind^PD=C$ and the subgroups~$N^D$) correspond to each other under conjugation. Finally, we split the integral into the difference of $J_C^T(f)$ and~$\tilde J_C^T(f)$, which is justified by Conjecture~\ref{expansion} and Hypothesis~\ref{hypoth2}(ii).
\end{proof}

\subsection{Ordering terms according to canonical par\-abolics}

We continue rewriting our formula for $J_C^T(f)$. The basic idea for the next step is that a sum over all elements of $G$ can be written as a sum over all parabolic subgroups $Q$ of partial sums over those elements whose canonical parabolic is~$Q$. This applies to~$K_G(x,x)$, but $J_C^T(f)$ also contains terms with $P\ne G$, which are indexed by cosets $\gamma N$ rather than elements. This is why the previous transformation was necessary.

\begin{lemma}\label{newJC}
Let $Q$ be the canonical parabolic of some element of~$C$. Under Hypotheses~\ref{NC}, \ref{hypoth2}(i) and~\ref{hypoth3} (the latter to be stated in the course of the proof), we have
\[
\tilde J_C^T(f)=\int_{Q\bs\G^1}\sum_{N'\subset Q}
\sum_P
\sum_{\substack{\gamma\in(C\cap Q\can)N'/N'\\
N^{[\gamma]}=N'\\
\gamma\in P\infl}}
\int_{\bN'} f(x^{-1}\gamma n'x)\,dn'\,\hat\tau_P^T(x)\,dx.
\]
Here the representative $\gamma$ is chosen in~$C\cap Q\can$, and $N^{[\gamma]}$ is a notation for $N^D$, where $D$ is the conjugacy class of $\gamma N$ in~$P/N$. If $C$ is unipotent, the condition $N'\subset Q$ can be sharpened to $N'\subset U'$.
\end{lemma}
Recall that the sets $Q\can$ and $U'$ were introduced in connection with Theorem~\ref{canpar}(iii). 

The summation over subgroups $N'$ may look weird. Of course, we need only consider subgroups which appear as~$N^D$ in Hypothesis~\ref{NC}. For unipotent conjugacy classes $D$ it is often the case that $N^D$ is the unipotent radical of a parabolic subgroup~$P^D$, namely the smallest parabolic which contains~$P$ and whose unipotent radical is contained in~$U'$. In this case, the sum over~$N'$ can be written as a sum over parabolics $P'$ containing~$Q$.
\begin{proof}
Let us fix a conjugacy class~$C$. The definition of $\tilde J_C^T(f)$ involves, for each~$P$, a sum over conjugacy classes $D$ in~$P/N$. We get the same result if we take the partial sum over those $D$ for which $N^D$ equals a given group $N'$ and add up those partial sums for all possible subgroups $N'$ of $G$.

For each $P$, $N'$ and $D$, we are now facing a sum over cosets $\gamma N'\in(C\cap P)N'/N'$. By the property~(ii) of~$N^D$ according to Hypothesis~\ref{NC}, the elements of $\gamma N'\cap C$ have the same canonical parabolic. Thus, we may similarly take the partial sum over those cosets for which that canonical parabolic equals a given group~$Q$ and add up the partial sums for all possible parabolics $Q$ in~$G$. As a result, we see that $\tilde J_C^T(f)$ equals
\[
\int_{G\bs\G^1}
\sum_P\sum_{N'}
\sum_{\substack{D\subset P/N\\N^D=N'}}
\sum_Q
\sum_{\substack{\gamma\in D'N'/N'\\
\gamma N'\cap C\subset Q\can}}\int_{\bN'}
f(x^{-1}\gamma n'x)\,dn'\,\hat\tau_P^T(x)\,dx.
\]

We want to move the summations over $Q$ and $N'$ leftmost. This is permitted under the following
\begin{hypoth}\label{hypoth3}
The integral
\[
\int_{G\bs\G^1}
\sum_Q\sum_{N'}\left|
\sum_P
\sum_{\substack{D\subset P/N\\N^D=N'}}
\sum_{\substack{\gamma\in D'N'/N'\\
\gamma N'\cap C\subset Q\can}}\int_{\bN'}
f(x^{-1}\gamma n'x)\,dn'\,\hat\tau_P^T(x)\right|\,dx
\]
is convergent.
\end{hypoth}

When $P$ is fixed, $D$ runs over a finite set, hence the order of the two inner summations is irrelevant. They can be written as a single sum over all pairs $(D,\gamma N')$ satisfying the conditions
\begin{enumerate}
\item[(i)] $N^D=N'$,
\item[(ii)] $\gamma\in D'N'/N'$, where $D'=C\cap DN$,
\item[(iii)] $\gamma N'\cap C\subset Q\can$.
\end{enumerate}
Since the set $\gamma N'\cap C$ is dense in~$\gamma N'$, quotients of its elements form a dense subset of~$N'$. Thefore, condition~(iii) can only be satisfied if $N'\subset Q$. If $C$ is unipotent, then $Q\cap C\subset U'$ by Theorem~\ref{canpar}(iii), and we must even have $N'\subset U'$.

Condition~(iii) also implies that
\begin{enumerate}
\item[($\mathrm{iii'}$)] $\gamma\in(C\cap Q\can)N'/N'$.
\end{enumerate}
Condition~(ii) shows that the representative $\gamma$ of $\gamma N'$ can be chosen in~$D'$, hence
\begin{enumerate}
\item[($\mathrm{ii'}$)] $\gamma\in P\infl$,
\end{enumerate}
and that $D$ is uniquely determined by $P$ and~$\gamma N'$. If we denote $N^D$ by $N^{[\gamma]}$, condition~(i) can be rewritten as
\begin{enumerate}
\item[($\mathrm{i'}$)] $N^{[\gamma]}=N'$.
\end{enumerate}

Conversely, suppose that we are given a coset $\gamma N'$ satisfying conditions ($\mathrm{i'}$), ($\mathrm{iii'}$) and~($\mathrm{ii'}$), the latter for a choice of $\gamma$ in~$C\cap Q\can$. Let $D$ be the conjugacy class of the image of $\gamma$ in~$P/N$. Then condition~(i) is satisfied, and hence $N'\subset P$. Therefore $D$ is independent of the choice of the representative~$\gamma$. Condition~($\mathrm{ii'}$) shows in view of Proposition~\ref{induction2}(iii) that $\gamma$ lies in $D'=\Infl^PD$, and in view of $\gamma\in C$ we have $D'=C\cap DN$. Thus condition~(ii) is satisfied. In hindsight we see that any other representative with the same properties lies in the $P$-orbit $D'$, hence it also satisfies condition~($\mathrm{ii'}$). Since $N'\subset Q$, we have $\gamma N'\subset Q$, and condition~(iii) follows by Theorem~\ref{canpar}(iii).

The equivalence of the two sets of conditions shows that $\tilde J_C^T(f)$ equals
\[
\int_{G\bs\G^1}\sum_Q\sum_{N'\subset Q}
\sum_P
\sum_{\substack{\gamma\in(C\cap Q\can)N'/N'\\
N^{[\gamma]}=N'\\
\gamma\in P\infl}}
\int_{\bN'} f(x^{-1}\gamma n'x)\,dn'\,\hat\tau_P^T(x)\,dx.
\]
If $Q$ is the canonical parabolic of some element of~$C$, we can obtain those of the other ones by conjugating with elements of~$Q\bs G$. Thus, rather than summing over all parabolics~$Q$, we may fix one of them, replace $x$ by $\delta x$ and insert a summation over $\delta\in Q\bs G$. Combining that summation with the exterior integral, we obtain our result.
\end{proof}
If we are allowed to interchange the summations over $P$ and $\gamma N'$, then $\tilde J_C^T(f)$ becomes
\[
\int_{Q\bs\G^1}\sum_{N'\subset Q}
\sum_{\gamma\in(C\cap Q\can)N'/N'}
\sum_{\substack{P\in\mathcal P_\gamma\infl\\
N^{[\gamma]}=N'}}
\int_{\bN'} f(x^{-1}\gamma n'x)\,dn'\,\hat\tau_P^T(x)\,dx.
\]
The set $\mathcal P_\gamma\infl$ was introduced in Theorem~\ref{induction2}(iv). According to our notational conventions in this section, $\mathcal P_\gamma\infl$ actually stands for $\mathcal P_\gamma\infl(F)$. The integral over~$\bN'$ is independent of~$P$ and can be extracted from the sum over~$P$.
\begin{cor}\label{newJCcor}
If only finitely many parabolic subgroups $P$ occur in the formula of Lemma~\ref{newJC}, then
\[
\tilde J_C^T(f)=\int_{Q\bs\G^1}\sum_{N'\subset Q}
\sum_{\gamma\in (C\cap Q\can)N'/N'}
\int_{\bN'}f(x^{-1}\gamma n'x)\,dn'\,
\chi_{\gamma N'}^T(x)\,dx,
\]
where
\[
\chi_{\gamma N'}^T(x)
=\sum_{\substack{P\in\mathcal P_\gamma\infl\\
N^{[\gamma]}=N'}}\hat\tau_P^T(x).
\]
\end{cor}
There is apparently no uniform argument justifying such an interchange of summations in general. Below, we will describe approaches to this problem and solutions in partial cases.

\subsection{Truncation classes}
\label{trclass}

We have sticked to geometric conjugacy classes so far because they afford a clean notion of induction. However, we are forced to split them up as evidence shows that various elements of the same class may behave differently in our formulas. For a moment, let us distinguish notationally between varieties and their sets of $F$-rational points again. While the sets $\mathcal P_\gamma\infl$ for various elements $\gamma$ in the same geometric conjugacy class $C$ are in bijection with each other under conjugation, the sets $\mathcal P_\gamma\infl(F)$ for $\gamma\in C(F)$ may be different if the elements are not $G(F)$-conjugate. This may happen even for unipotent classes.

For the lack of a better idea, we call elements $\gamma_1$, $\gamma_2$ of $G(F)$ truncation equivalent if they belong to the same geometric conjugacy class and if every inner automorphism mapping $\gamma_1$ to $\gamma_2$ will map $\mathcal P_{\gamma_1}\infl(F)$ onto~$\mathcal P_{\gamma_2}\infl(F)$. This is an equivalence relation, because conjugate $F$-rational parabolics are $G(F)$-conjugate. The equivalence classes for this relation will be called truncation classes.

It follows from the definition that an element of $P\infl(F)$ will also belong to $P^{\prime\,\mathrm{infl}}(F)$ for any parabolic $P'$ containing~$P$. Thus, if $P\in\mathcal P_\gamma\infl(F)$, then $P'\in\mathcal P_\gamma\infl(F)$. The inclusion relation among the sets $\mathcal P_\gamma\infl(F)$ therefore defines a partial order on the finite set of truncation classes $O$ in a given geometric conjugacy class~$C(F)$.

In order to split up $J_C^T(f)$ into contributions of truncation classes, we have to do so for the kernel functions~$K_{P,C}$. Thus, to every $F$-rational coset $\gamma N$ meeting~$C$, we have to assign a truncation class. Evidence suggests that we should pick the minimal truncation class $O$ meeting $\gamma N(F)$. Its uniqueness will have to be proved.

We expect that this definition produces the correct grouping of terms so that all of the previous discussion applies to the resulting distributions $J_O^T(f)$. In Lemma~\ref{newJC}, we have to replace $P\infl(F)$ by the subset $P^{\textrm{min\,infl}}(F)$ of elements whose truncation class is minimal among those meeting~$P$. In the corollary, we have to replace $\mathcal P_\gamma\infl(F)$ by the subset $\mathcal P_\gamma^{\textrm{min\,infl}}(F)$ of those parabolics~$P$ for which $\gamma\in P^{\textrm{min\,infl}}(F)$.

\section{Relation to zeta integrals}
\subsection{Damping factors}
\label{dampsection}

A classical artifice to make the full integral-sum in Lemma~\ref{newJC} (or its analogue for a truncation class~$O$) absolutely convergent, primarily in the case of unipotent orbits, is to insert a damping factor $e^{-\langle\lambda,H_Q(x)\rangle}$ into the integrand of~$\tilde J_O^T(f)$, where $\lambda$ is a complex-valued linear function on~$\af_Q$, to obtain a distribution $J_O^T(f,\lambda)$. This idea goes back to Selberg and has been applied in~\cite{A-rk1}, \cite{Ho0}, \cite{HoWa} and other papers. In all cases considered so far, the following was satisfied.
\begin{hypoth}\label{damp}
Let $O$ be a truncation class. Then the integral-sum
\begin{multline*}
J_O^T(f,\lambda)=
\int_{Q\bs\G^1}
e^{-\langle\lambda,H_Q(x)\rangle}
\sum_{N'\subset Q}\\
\sum_{\gamma\in (O\cap Q\can)N'/N'}
\int_{\bN'}f(x^{-1}\gamma n'x)\,dn'\,
\chi_{\gamma N'}^T(x)\,dx
\end{multline*}
is absolutely convergent for $\Re\lambda$ in a neighbourhood of zero. If a group $N'$ occurring here is the unipotent radical of parabolic subgroup~$P'$, then its contribution $J_{O,P'}^T(f,\lambda)$ is absolutely convergent for $\Re\lambda$ in a certain positive chamber and has a meromorphic continuation to a domain including the point $\lambda=0$.
\end{hypoth}
Assume that all subgroups $N'$ occurring in $J_O^T(f,\lambda)$ are unipotent radicals of parabolics~$P'$, so that the sum over $N'$ is finite. If we choose $\lambda_0$ such that $\C\lambda_0$ is not contained in the singular set of any of these Dirichlet series, then the value of the regular function $J_O^T(f,\lambda)$ at $\lambda=0$ is
\[
J_O^T(f)=\sum_{P'\supset Q}
\mathop{\textnormal{f.p.}}\limits_{z=0}J_{O,P'}^T(f,z\lambda_0),
\]
where \mbox{f.p.} denotes the finite part in the Laurent expansion.

We can remove the dependence on the group~$Q$ by considering the canonical parabolic subgroups $Q(\gamma)$ of all elements $\gamma\in O$ and a family $\lambda$ of linear functions $\lambda(Q)$ on the spaces~$\af_{Q}$ which is coherent in the sense that $\lambda(\delta^{-1}Q\delta)=\lambda(Q)\circ\Ad\delta$ for all $\delta\in G(F)$. In the special case $P'=G$ we get
\[
J_{O,G}^{T,T'}(f,\lambda)=
\int_{G\bs\G^1}
\sum_{\gamma\in O}
e^{-\langle\lambda(Q(\gamma)),H_{Q(\gamma)}^{T'}(x)\rangle}
f(x^{-1}\gamma x)\,
\chi_\gamma^T(x)\,dx,
\]
which depends on an additional truncation parameter $T'$ that was fixed before by requiring $T_Q=0$ for the chosen parabolic~$Q$.

In order to explain how the Hypothesis gives rise to weighted orbital integrals, we need some preparation. One fixes a finite set $S$ of places of~$F$ including all archimedean ones and decomposes the ring of adeles as a direct product $\A=F_S\A^S$ and the group as $G(\A)=G(F_S)G(\A^S)$. Then every function $f\in C_c^\infty(G(F_S))$ can be extended to $G(\A)$ by multiplying it with the characteristic function of $\bK^S=\bK\cap G(\A^S)$. The group $G(F_S)^1$ acts from the right on $G(F)\bs G(\A)^1/\bK^S$ with finitely many orbits, and the stabiliser of the orbit with representative $g\in G(\A^S)^1$ is the $S$-arithmetic subgroup
\[
\Gamma_g=\{\delta_S\mid\delta\in G(F),\,\delta^S\in g\bK^S g^{-1}\},
\]
where $\delta_S$ and $\delta^S$ are the images of $\delta$ in $G(F_S)$ and~$G(\A^S)$, resp. Its subset
\[
O_g=\{\gamma_S\mid\gamma\in O,\,\gamma^S\in g\bK^S g^{-1}\}
\]
is invariant under conjugacy. Let us restrict to the case $P'=G$ for simplicity. If we substitute $x=gy$ with $y\in G(F_S)$, then the integrand vanishes unless $\gamma^S\in g\bK^S g^{-1}$, and the distribution $J_{O,G}^{T,T'}(f,\lambda)$ becomes
\[
\sum_g\int_{\Gamma_g\bs G(F_S)^1}
\sum_{\gamma\in O_g}
e^{-\langle\lambda(Q(\gamma)),H_{Q(\gamma)}^{T'(g)}(y)\rangle}
f(y^{-1}\gamma y)\,
\chi_\gamma^{T(g)}(y)\,dy,
\]
where $T(g)_P=T_P-H_P(g)$, because $H_P(gy)=H_P(g)+H_P(y)$. Ordering the elements according to canonical parabolics, we obtain the $S$-arithmetic version of the original formula
\begin{multline*}
J_{O,G}^{T,T'}(f,\lambda)=\sum_g\sum_{[Q]_{\Gamma_g}}
\int_{\Gamma_g\cap Q(F)\bs G(F_S)^1}
e^{-\langle\lambda(Q),H_{Q}^{T'(g)}(y)\rangle}\\
\sum_{\gamma\in O_g\cap Q\can}
f(y^{-1}\gamma y)\,
\chi_\gamma^{T(g)}(y)\,dy.
\end{multline*}
Here we split the inner sum into subsums over $\Gamma_g\cap Q(F)$-conjugacy classes in $O_g\cap Q\can(F)$, written as sums over $\Gamma_g\cap Q^\gamma(F)\bs \Gamma_g\cap Q(F)$, which can be combined with the integral once the summation over the classes has been moved outside:
\begin{multline*}
J_{O,G}^{T,T'}(f,\lambda)=\sum_g\sum_{[Q]_{\Gamma_g}}\sum_{[O_g\cap Q\can]_{\Gamma_g\cap Q}}
\operatorname{vol}(\Gamma_g\cap Q^\gamma(F)\bs Q^\gamma(F_S)^1)\\
\int_{Q^\gamma(F_S)^1\bs G(F_S)^1}
e^{-\langle\lambda(Q),H_{Q}^{T'(g)}(y)\rangle}
f(y^{-1}\gamma y)\,
\chi_\gamma^{T(g)}(y)\,dy.
\end{multline*}
If $\gamma=q^{-1}\gamma_0q$ with $q\in Q(F_S)$, then the substitution $qy=z$ transforms the integral into its analogue for $\gamma_0$ times $\exp\langle\lambda(Q),H_Q(q)\rangle$. The set $C(F_S)$ consists of finitely many $G(F_S)^1$-conjugacy classes, and each of them intersects $C(F_S)\cap Q\can(F_S)$ in a $Q(F_S)$-conjugacy class. Choosing representatives $\gamma_0\in Q\can(F_S)$, we can write
\begin{multline*}
J_{O,G}^{T,T'}(f,\lambda)=\sum_g\sum_{[Q]_{\Gamma_g}}
\sum_{[\gamma_0]_S}\zeta_G(g,Q,\gamma_0,\lambda)\\
\int_{Q^{\gamma_0}(F_S)^1\bs G(F_S)^1}
e^{-\langle\lambda(Q),H_{Q}^{T'(g)}(y)\rangle}
f(y^{-1}\gamma_0 y)\,
\chi_{\gamma_0}^{T(g)}(y)\,dy,
\end{multline*}
where $\zeta_G(g,Q,\gamma_0,\lambda)$ is a certain Dirichlet series in the variable~$\lambda$. Since the parabolics $Q$ are conjugate,  suitable substitutions reduce the integrals to multiples of
\[
J_G^T(\gamma,f,\lambda)=
\int_{Q^\gamma(F_S)^1\bs G(F_S)^1}
e^{-\langle\lambda,H_Q(y)\rangle}
f(y^{-1}\gamma y)\,
\chi_\gamma^T(y)\,dy
\]
for a fixed~$Q$, where we have set $T'_Q=0$, and we get
\[
J_{O,G}^T(f,\lambda)=
\sum_{[\gamma]_S}\zeta_G(S,\gamma,\lambda)J_G^T(\gamma,f,\lambda).
\]
We do not indicate the dependence of the weighted orbital integral on $S$ as this information is encoded in the argument~$f$. Similarly one can show that, if $N'$ is the unipotent radical of $P'\supset Q$,
\[
J_{O,P'}^T(f,\lambda)=
\sum_{[\gamma N']_S}
\zeta_{P'}(S,\gamma,\lambda)J_{P'}^T(\gamma N',f,\lambda),
\]
where $[\gamma N']_S$ runs through the set of $Q(F_S)$-conjugacy classes in the quotient $(C(F_S)\cap Q\can(F_S))N'(F_S)/N'(F_S)$,
\begin{multline*}
J_{P'}^T(\gamma N',f,\lambda)=
\int_{Q^{\gamma N'}(F_S)^1\bs G(F_S)^1}
e^{-\langle\lambda,H_Q(x)\rangle}\\
\int_{N'(F_S)}f(x^{-1}\gamma n' x)\,dn'\,\chi_{\gamma N'}^T(x)\,dx
\end{multline*}
with the notation $Q^{\gamma N'}$ for the stabiliser of~$\gamma N'$ in~$Q$, and where $\zeta_{P'}(S,\gamma,\lambda)$ are certain Dirichlet series. They depend only on the coset~$\gamma N'(F_S)$, but we prefer to write them as functions of $\gamma$ for reasons that become clear at the end of section~\ref{prin}.

If these zeta functions can be meromorphically continued, one obtains a formula with explicit weight factors, because the Laurent expansion of a product can be expressed by those of its factors. In section~\ref{prin} we will carry this out for the principal unipotent conjugacy class.

One has to regroup the result in terms of conjugacy classes of Levi subgroups~$M$ and to relate the resulting explicit weighted orbital integrals to Arthur's distributions $J_M(\gamma,f)$ if one wishes to compute the coefficients $a^M(S,\gamma)$ in the fine geometric expansion~\cite{A-mix}. So far, this has only been done in special cases (like in~\cite{HoWa}) by an ad-hoc computation.

\subsection{Reduction to vector spaces}
\label{redtoPV}

Let $O$ be a truncation class in a unipotent conjugacy class $C$ and $Q$ the canonical parabolic of one of its elements. We perform a further transformation of $J_O^T(f,\lambda)$ which is in a way contrary to that in Lemma~\ref{int2sum}, because this time we are replacing sums by integrals.

The group~$Q/U$, which can be identified with any Levi subgroup $L$ of~$Q$ defined over~$F$, acts on the group $V=U'/U''$ by conjugation. This action is linear if we endow $U'$ and $U''$ with the structure of vector spaces defined over~$F$ using the exponential maps. By Theorem~\ref{canpar}(ii), $V$ is an $F$-regular prehomogeneous vector space, and the generic orbit is~$C\cap U'/U''$. For each parabolic subgroup $P'\supset Q$, we have the vector subspace $V_{P'}=N'U''/U''$ and quotient space $V^{P'}=V/V_{P'}$, both prehomogeneous by Theorem~\ref{prehom}(i).

Now we switch back to the simplified notation for adelic and rational points of varieties introduced in section~\ref{trunc}. For each $L$-invariant subquotient $W$ of~$U$, we denote by $\delta_W$ the modular character for the action of $\mathbf L$ on~$\mathbf W$ by inner automorphisms. It can be interpreted as an element of~$\af_L^*$, so that $\delta_W(l)=e^{\langle\delta_W,H_L(l)\rangle}$. We have to be cautious the tradition of the adelic trace formula imposes upon us the right action by inverses of inner automorphisms.

\begin{hypoth}\label{convZeta}
In the situation of Hypothesis~\ref{damp}, for a unipotent truncation class $O$ and every Schwarz-Bruhat function $\varphi$ on~$\V$, the integral-sum
\begin{multline*}
Z_O^T(\varphi,\lambda)=
\int_{L\bs\mathbf L\cap\G^1}
e^{-\langle\lambda+\delta_{U/U''},H_L(l)\rangle}
\sum_{N'\subset Q}\\
\sum_{\nu\in(O\cap U')N'/U''N'}
\int_{\bN'\mathbf U''/\mathbf U''}\varphi(l^{-1}\nu n'l)\,dn'\,\chi_{\nu N'}^T(l)\,dl,
\end{multline*}
is absolutely convergent for~$\Re\lambda$ in a neighbourhood of the closure of~$(\af_Q^*)^+$.

If a group $N'$ occurring here is the unipotent radical of a parabolic subgroup~$P'$, then for every Schwarz-Bruhat function $\psi$ on~$\V^{P'}$, the truncated zeta integral
\begin{multline*}
Z_{O,P'}^T(\psi,\lambda)=
\int_{L\bs\mathbf L\cap\G^1}
e^{-\langle\lambda+\delta_{U/N'U''},H_L(l)\rangle}\\
\sum_{\nu\in(O\cap U')N'/U''N'}
\psi(l^{-1}\nu l)\,\chi_{\nu N'}^T(l)\,dl,
\end{multline*}
is absolutely convergent for~$\Re\lambda\in(\af_Q^*)^+$ and extends meromorphically to a neighbourhood of the closure of that domain.
\end{hypoth}
Let us look at the special case $P'=G$. As in Theorem~\ref{zetaInt}, such integrals usually converge when the parameter in the exponent is $\lambda+\delta_V$ with $\Re\lambda$ positive on the chamber $\af_Q^+$. Our parameter shift differs by~$\delta_{U/U'}$, and if this point is contained in the domain of convergence, no terms with $P'\ne G$ are needed for regularisation.

Due to the restriction to~$\G^1$, the usual convergence condition takes the form $X(L_\nu)_F=X(G)_F$ or, equivalently, $A_{L_\nu}=A_G$, where $A_G$ denotes the largest $F$-split torus in the centre of~$G$. It may be violated, but the truncation function~$\chi_\nu^T$ should save the convergence.

\begin{lemma}\label{JasPI}
Under Hypotheses \ref{damp} and~\ref{convZeta}, we have
\[
J_O^T(f,\lambda)=Z_O^T(f_V,\lambda),\qquad
J_{O,P'}^T(f,\lambda)=Z_{O,P'}^T(f_V^{P'},\lambda),
\]
where
\[
f_V(v)=\int_\bK\int_{\mathbf U''} f(k^{-1}vu''k)\,du''\,dk
\]
is a smooth compactly supported function on~$\V$ and, for each Schwarz-Bruhat function $\varphi$ on~$\V$,
\[
\varphi^{P'}(v)=\int_{\V_{P'}} \varphi(v v')\,dv'
\]
is a Schwarz-Bruhat functions on~$\V^{P'}$.
\end{lemma}
\begin{proof}
For each $\gamma\in O\cap U'$, the map $U_\gamma\bs U\to\gamma U''$ given by $\delta\mapsto\delta^{-1}\gamma\delta$ is an isomorphism due to the representation theory of~$\mathfrak{sl}_2$. An isomorphism between affine $F$-varieties induces a bijection beween their sets of $F$-rational points, and the set of $F$-rational points of $U'/U''$ is $U'(F)/U''(F)$. Since $O\cap U(F)$ is normalised by~$U(F)$, all elements of a $U''(F)$-coset in $U'(F)$ belong to the same truncation class. Writing again $U$ for $U(F)$ etc., we get for a finitely supported functions $g$ on~$U'$ and $h$ on~$\nu U''$
\[
\sum_{\gamma\in O\cap U'} g(\gamma)
=\sum_{\nu\in O\cap U'/U''} \sum_{\eta\in U''} g(\nu\eta),\quad
\sum_{\eta\in U''} h(\nu\eta)
=\sum_{\delta\in U_\nu\bs U} h(\delta^{-1}\nu\delta).
\]
This argument also works if we replace $U'$ by $U'/N'$ and $U''$ by its image $U''N'/N'$ in that quotient. It shows that, for $g$ on $V$,
\[
\sum_{\gamma\in(O\cap U')N'/N'} g(\gamma)
=\sum_{\nu\in(O\cap U')N'/U''N'} \sum_{\delta\in U_{\nu N'}\bs U} g(\delta^{-1}\nu\delta).
\]
As a byproduct, we see that the sum in the Lemma is well defined. These identities have adelic versions, too, of which we only need
\[
\int_{\U''\bN'/\bN'} h(\nu u')\,du'
=\int_{\U_{\nu N'}\bs\U} h(u^{-1}\nu u)\,du
\]
for continuous compactly supported functions $h$ on~$\nu\U''\bN'/\bN'$.

By definition, $J_O^T(f,\lambda)$ is given by the expression
\[
\int_{Q\bs\G^1}
e^{-\langle\lambda,H_Q(x)\rangle}
\sum_{N'\subset Q}
\sum_{\gamma\in (O\cap U')N'/N'}
\int_{\bN'}f(x^{-1}\gamma n'x)\,dn'\,
\chi_{\gamma N'}^T(x)\,dx.
\]
Upon applying the identity we have just proved, the inner sum becomes
\[
\sum_{\nu\in(O\cap U')N'/U''N'} \sum_{\delta\in U_{\nu N'}\bs U}
\int_{\bN'}f(x^{-1}\delta^{-1}\nu\delta n'x)\,dn'\,
\chi_{\delta^{-1}\nu\delta N'}^T(x).
\]
Note that
\[
\chi_{\delta^{-1}\nu\delta N'}^T(x)=\chi_{\nu N'}^T(\delta x).
\]
Substituting $\delta n'=n\delta$, decomposing the exterior integral according to~$\G=\U\bL\bK$ and combining the integral over $U\bs\U$ with the sum over~$U_{\nu N'}\bs U$, we get
\begin{multline*}
J_O^T(f,\lambda)=
\int_\bK\int_{L\bs\bL\cap\G^1}
e^{-\langle\lambda,H_L(l)\rangle}
\sum_{N'\subset Q}
\sum_{\nu\in(O\cap U')N'/U''N'} \\
\int_{U_{\nu N'}\bs\U}
\int_{\bN'}f(k^{-1}l^{-1}u^{-1}\nu nulk)\,dn\,
\chi_{\nu N'}^T(ul)\,du\,\delta_U(l^{-1})\,dl\,dk.
\end{multline*}
We split the integral over $U_{\nu N'}\bs\U$ into integrals over $\U_{\nu N'}\bs\U$ and $U_{\nu N'}\bs\U_{\nu N'}$. The latter one drops out, since the integral over $\bN'$ as a function of~$u$ and the function $\chi_{\nu N'}^T$ are left-invariant under $\U_{\nu N'}$, and with the usual normalisation, the measure of $U_{\nu N'}\bs\U_{\nu N'}$ equals~$1$.

For all $P$ containing~$\gamma$, $H_P$ is left $\U$-invariant by Theorem~\ref{canpar}(iv), hence so is $\chi_{\nu N'}^T$ and can be extracted from the integral over~$u$. Substituting $nu=un'$, applying the adelic version of the above identity to
\[
h(\nu u')=\int_{\bN'}f(k^{-1}l^{-1}\nu u'n'lk)\,dn'
\]
and combining the integrals over $u'$ and~$n'$, we obtain
\begin{multline*}
J_O^T(f,\lambda)=
\int_\bK\int_{L\bs\bL\cap\G^1}
e^{-\langle\lambda,H_L(l)\rangle}\delta_U(l^{-1})
\sum_{N'\subset Q}\\
\sum_{\nu\in(O\cap U')N'/U''N'}
\int_{\U''\bN'}f(k^{-1}l^{-1}\nu nlk)\,dn\,\chi_{\nu N'}^T(l)\,dl\,dk.
\end{multline*}
It remains to move the integral over $\bK$ under the sum, to split the integral over $\U''\bN'$ into integrals over $n'\in\bN'\mathbf U''/\mathbf U''$ and~$u''\in\U''$ and to substitute $u''l=lu'$, which produces a factor $\delta_{U''}(l)$. If we treat only the contribution from a fixed group~$P'$, we may substitute $n'l=lv'$ in the integral over $\V_{P'}=\bN'\mathbf U''/\mathbf U''$, which produces a factor $\delta_{N'U''/U''}(l)$ and allows us to express everything in terms of~$\psi=\varphi^{P'}$.
\end{proof}

\section{The principal unipotent contribution}
\label{prin}
\subsection{Reduction to the trivial parabolic}

A final formula can be obtained for the contribution of the principal unipotent conjugacy class in $G$, which we denote by~$G\prin$. Let $\gamma$ be an element. Its canonical parbolic $Q$ is then a minimal parabolic, for which we choose a Levi component~$L$. By Theorem~\ref{canpar}, $G\prin\cap Q\can$ is a dense $Q$-conjugacy class in~$U=U'$. The set $\mathcal P_\gamma\infl$ consists of all parabolics $P$ containing~$Q$. For each of them, $G\prin\cap P$ is dense in $M\prin N$, where $N$ is the unipotent radical of $P$ and $M$ the Levi component containing~$L$. By Theorem~\ref{prehom}, $\gamma N$ is prehomogeneous under the action of~$Q_{\gamma N}\subset P_{\gamma N}$ with generic orbit contained in the $Q$-conjugacy class of~$\gamma$, hence in~$Q\can$. Therefore we can take $N^{M\prin}=N$, and one could even show that this is the only choice satisfying Hypothesis~\ref{NC}.

Thus, the definition given in Hypothesis~\ref{damp} simplifies to
\begin{multline*}
J_{G\prin,P}^T(f,\lambda)
=\int_{Q\bs\G^1}e^{-\langle\lambda,H_Q(x)\rangle}\\
\sum_{\gamma\in (G\prin\cap Q)N/N}\int_\bN f(x^{-1}\gamma n' x)\,dn'\,\hat\tau_P^T(x)\,dx,
\end{multline*}
which depends only on the restriction of $\lambda$ to~$\af_Q^G$. These distributions vor various $P$ can be expressed in terms of the one with $P=G$ (which does not depend on the truncation parameter), in which the ambient group $G$ is replaced by~$M$ (indicated by a superscript~$M$).
\begin{lemma}\label{Jprin}
If Hypothesis~\ref{damp} applies to principal unipotent orbits, then
\[
J_{G\prin,P}^T(f,\lambda)=\epsilon_P
J_{M\prin,M}^M(f^P,\lambda^P)\theta_P^T(\lambda)^{-1},
\]
where
\[
f^P(m)=\int_\bK\int_\bN f(k^{-1}mnk)\,dn\,dk
\]
is a compactly supported smooth function on~$\mathbf M^1$ and, in the notation of section~\ref{GQfam},
$\theta_P^T(\lambda)=
e^{\langle\lambda,T_P\rangle}\theta_P(\lambda)$.
\end{lemma}
\begin{proof}
The natural map $M\prin\cap Q\to(G\prin\cap Q)N/N$ is a bijection, and with the usual integration formula for the decomposition $\G=\bN\mathbf M\bK$, the above expression can be written as
\begin{multline*}
\int_\bK\int_{M\cap Q\bs\mathbf M\cap\G^1}\int_{N\bs\bN}
e^{-\langle\lambda,H_{M\cap Q}(m)\rangle}\\
\sum_{\gamma\in M\prin\cap Q}\int_\bN
f(k^{-1}m^{-1}n^{-1}\gamma n'nmk)\,dn'\,\hat\tau_P^T(m)\,dn\,\delta_N(m)^{-1}dm\,dk,
\end{multline*}
where the integral over $N\bs\bN$ drops out. Now we substitute $n'm=mn$, thereby cancelling the factor $\delta_N(m)$, and split the integral over $M\cap Q\bs\mathbf M\cap\G^1$ into integrals over $\mathbf M^1\bs\mathbf M\cap\G^1\cong\af_M^G=\af_P^G$ and $M\cap Q\bs\mathbf M^1$. Since the elements $\gamma\in M$ act trivially on $\af_M^G$, we get
\begin{multline*}
\int_\bK\int_{M\cap Q\bs\mathbf M^1}
e^{-\langle\lambda,H_{M\cap Q}(m)\rangle}
\sum_{\gamma\in M\prin\cap Q}\int_\bN
f(k^{-1}m^{-1}\gamma mnk)\,dn\,dm\,dk\\
\epsilon_P\int_{\af_P^G}
e^{-\langle\lambda,H\rangle}\hat\tau_P(H-T_P)\, dH,
\end{multline*}
where we have used that $\hat\tau_P^T(x)=\epsilon_P\hat\tau_P(H_P^T(x))$.
\end{proof}

\subsection{Singularities of zeta integrals}

Under Hypotheses \ref{damp} and~\ref{convZeta}, Lemma~\ref{JasPI} shows that the distributions on both sides of the equality
\[
J_{G\prin}^T(f,\lambda)=\sum_{P\supset Q}J_{G\prin,P}^T(f^P,\lambda)
\]
as well as in Lemma~\ref{Jprin} can be expressed in terms of zeta integrals. Since the two possible interpretations of $f_V^P$ coincide, there will be parallel formulas for zeta integrals. We prove them unconditionally.
\begin{lemma}\label{singZ}
For every Schwarz-Bruhat function $\varphi$ on~$\V$ and $\lambda\in(\af_Q^*)^+$, we have
\begin{align*}
Z_{G\prin}^T(\varphi,\lambda)&=\sum_{P\supset Q}Z_{G\prin,P}^T(\varphi^P,\lambda),\\
Z_{G\prin,P}^T(\varphi,\lambda)&=\epsilon_P
Z_{M\prin,M}^M(\varphi^P,\lambda^P)\theta_P^T(\lambda)^{-1}.
\end{align*}
These functions are defined by convergent integral-sums for $\Re\lambda\in(\af_Q^*)^+$ and extend meromorphically to all $\lambda$ with $\Re\lambda$ in a neighbourhood of zero. The function $Z_{G\prin}^T(\varphi,\lambda)$ is holomorphic there.
\end{lemma}
\begin{proof}
The space $V$ is the direct sum of the prehomogeneous vector spaces $V^P$ corresponding to the minimal parabolics $P$ properly containing~$Q$. For each such~$P$, the space $V^P$ is the isomorphic image, under the exponential map, of a root space for a fundamental root $\alpha$ of the maximal split torus in $L$, and we get a bijection between the set of $F$-irreducible summands of $V$ and the set $\Delta_Q$. The group $L$ is $F$-anisotropic modulo centre, and its basic characters corresponding to the relative invariants of $V$, when restricted to the maximal split torus in the centre of~$L$, are nothing but the elements of $\Delta_Q$. This prompts us to write $\lambda$ as a linear combination of those fundamental roots with certain coefficients~$s_\alpha$. These coefficients are then the values of $\lambda$ on the elements of the dual basis, viz.\ the fundamental coweights~$\check\varpi$.

The distribution $Z_{G\prin,G}(\varphi,\lambda)$ is a zeta integral without truncation on the prehomogeneous vector space $V$. By Theorem~\ref{zetaInt}, it converges absolutely if $\langle\Re\lambda,\check\varpi\rangle>0$ for all $\varpi\in\hat\Delta_Q$ and extends meromorphically to the whole space. Its singularities for $\Re\lambda$ in some neighbourhood of zero are at most simple poles along the hyperplanes where $\langle\lambda,\check\varpi\rangle=0$ for some~$\varpi\in\hat\Delta_Q$, and the multiple residue at any point $\lambda_0$ in that neighbourhood is described as follows. If $P$ is the smallest parabolic containing~$Q$ such that $\lambda_0$ vanishes on~$\af_P^G$ (i.~e., the singular hyperplanes containing~$\lambda_0$ are indexed by~$\hat\Delta_P$), then
\[
\lim_{\lambda\to\lambda_0}
Z_{G\prin,G}(\varphi,\lambda)
\hat\theta_P(\lambda)
=Z_{M\prin,M}^M(\varphi^P,\lambda_0).
\]
An argument as in the proof of Lemma~\ref{Jprin} shows the second asserted identity, which provides the convergence and meromorphic continuation of its left-hand side. The manipulations at the end of the proof of Lemma~\ref{JasPI} are now valid unconditionally, thus proving the first identity for $\lambda$ in the domain of convergence and hence for the meromorphically continued functions.

The theory of $(G,Q)$-families cannot be applied to meromorphic functions. One may remove the singularities of the zeta integral at $\lambda=0$ either by multiplying with linear functions or by subtracting the principal part. The first method leads to the modified distribution
\[
\tilde Z_{G\prin,G}(\varphi,\lambda)=
Z_{G\prin,G}(\varphi,\lambda)
\hat\theta_Q(\lambda)
\]
and its analogues for Levi subgroups. Since the elements of the dual basis of $\Delta_Q^P$ are the projections of the $\check\varpi$ with $\varpi\in\hat\Delta_Q\setminus\hat\Delta_P$ onto~$\af_Q^P$, it follows that 
\[
\lim_{\lambda\to\lambda_0}
\tilde Z_{G\prin,G}(\varphi,\lambda)
=\tilde Z_{M\prin,M}^M(\varphi^P,\lambda_0).
\]
This shows that the functions $c_P(\lambda)=e^{-\langle\lambda,T_P\rangle}\tilde Z_{M\prin,M}^M(\varphi^P,\lambda^P)$ make up a $(G,Q)$-family, which is a product of a frugal and a cofrugal one. We can rewrite our formula for the principal unipotent contribution tautologically as
\[
Z_{G\prin}^T(\varphi,\lambda)=\sum_{P\supset Q}
\epsilon_P
\tilde Z_{M\prin,M}^M(\varphi^P,\lambda^P)
\hat\theta_Q^P(\lambda)^{-1}
\theta_P^T(\lambda)^{-1}.
\]
Now the regularity of the right-hand side for $\Re\lambda$ in a neighbourhood of zero follows from Lemma~\ref{GQ}.
\end{proof}

Let us discuss the second method of removing singularities that was mentioned in the proof. Note that the principal part of a meromorphic function on a complex space is not invariantly defined. Thus, we exploit the $L$-invariant splitting $V=V_P\oplus V_R$ of our prehomogeneous vector space valid for each pair of parabolics $P$ and $R$ containing~$Q$ for which $\Delta_Q$ is the disjoint union of $\Delta_Q^P$ and~$\Delta_Q^R$. Although the image of $L$ in $\operatorname{Aut}(V)$ need not split accordingly, that of its centre does, leading to the decomposition $\af_Q^G=\af_P^G\oplus\af_R^G$. Since $R$ is determined by $P$ and~$Q$, we denote $\lambda_R$ by $\lambda^{P/Q}$, which may serve as an argument for~$Z_{M\prin,M}^M$, because both $\af_R^G$ and $\af_Q^P$ are canonically isomorphic to~$\af_Q/\af_P$. (This approach is dual to the one applied in the proof of Lemma~6.1 in~\cite{A-inv}.) We define the second modified distribution as
\[
\tilde Z_{G\prin}(\varphi,\lambda)
=\sum_{P\supset Q}\epsilon_P
Z_{M\prin,M}^M(\varphi^P,\lambda^{P/Q})
\hat\theta_P(\lambda)^{-1}
\]
(without the additional subscript $G$), which is also holomorphic for $\Re\lambda$ a neighbourhood of zero, because the poles along each singular hyperplane cancel. We have a version of this distribution for the Levi component $M'$ of every parabolic $P'\supset Q$ in the role of~$G$, and by induction we can easily prove the converse relation
\[
Z_{G\prin,G}
(\varphi,\lambda)
=\sum_{P'\supset Q}
\tilde Z_{M^{\prime\text{prin}}}^{M'}
(\varphi^{P'},\lambda^{P'/Q})
\hat\theta_{P'}(\lambda)^{-1}.
\]
Plugging its relative version into the formula for $Z_{G\prin,P}^T(\varphi,\lambda)$ and summing over~$P$, we obtain after a change of summation a second formula
\[
Z_{G\prin}^T(\varphi,\lambda)=\sum_{P'\supset Q}
\sum_{P\supset P'}
\epsilon_P
\tilde Z_{M^{\prime\text{prin}}}^{M'}
(\varphi^{P'},(\lambda^P)^{M\cap P'/M\cap Q})
\hat\theta_{P'}^P(\lambda)^{-1}
\theta_P^T(\lambda)^{-1}.
\]
The functions $c_P(\lambda_{P'})=e^{-\langle\lambda,T_P\rangle}\tilde Z_{M^{\prime\text{prin}}}^{M'}
(\varphi^{P'},(\lambda^P)^{M\cap P'/M\cap Q})$, for fixed $P'$ and~$\lambda^{P'}$, constitute a $(G,P')$-family, hence the inner sum is holomorphic in~$\lambda_{P'}$ by Lemma~\ref{GQ}. It is actually holomorphic in~$\lambda$, because the family depends holomorphically on~$\lambda^{P'}$ in the obvious sense. For $P'=G$, it reduces to $\tilde Z_{G\prin}(\varphi,\lambda)$, while the contribution of $P'=Q$ converges to
\[
Z_{L\prin}^L(\varphi^Q)\int_{\af_Q^G}\Gamma_Q'(H,T_Q)\,dH
\]
as $\lambda\to0$ by Lemma~2.2 of~\cite{A-inv}.

\subsection{Explicit weight factors}

Now let $f\in C_c^\infty(G(F_S)^1)$ for a finite set $S$ of places. As in subsection~\ref{dampsection}, we have the expansion
\[
J_{G\prin,P}^T(f,\lambda)=
\sum_{[\gamma' N]_S}\zeta_P(S,\gamma',\lambda)J_P^T(\gamma' N,f,\lambda)
\]
for any $P\supset Q$, where $[\gamma' N]_S$ runs through the $Q(F_S)$-conjugacy classes in~$(G\prin(F_S)\cap Q(F_S))N(F_S)/N(F_S)$.
Here $\zeta_P(S,\gamma',\lambda)$ is a certain zeta function associated to the prehomogeneous vector space~$V^P$, and the weighted orbital integral $J_P^T(\gamma' N,f,\lambda)$ is given by
\[
\int_{Q^{\gamma' N}(F_S)^1\bs G(F_S)^1}
e^{-\langle\lambda,H_Q(x)\rangle}\int_{N(F_S)}f(x^{-1}\gamma' n x)\,dn\,\hat\tau_P^T(x)\,dx.
\]
As in the proof of Lemma~\ref{Jprin}, we see that
\[
J_P^T(\gamma' N,f,\lambda)=\epsilon_P
J_M^M(\gamma' N,f^P,\lambda^P)
\theta_P^T(\lambda)^{-1},
\]
where the superscript $M$ indicates the analogue of the distribution for $M$ in place of~$G$. The latter is holomorphic, hence the zeta function is responsible for the remaining singularities of the product. We remove them by setting
\[
\tilde\zeta_P(S,\gamma',\lambda)=\zeta_P(S,\gamma',\lambda)\hat\theta_Q^P(\lambda).
\]
The preimage in $G(F_S)\prin\cap Q(F_S)$ of a $Q(F_S)$-orbit $[\gamma' N]_S$ consists of several $Q(F_S)$-orbits. Using the isomorphism $N_\gamma\bs N\times N_\gamma\to\gamma N$ as in the proof of Lemma~\ref{JasPI}, we get
\[
J_M^M(\gamma'N,f^P,\lambda^P)=\sum_{[\gamma]_S:[\gamma N]_S=[\gamma' N]_S}J_P(\gamma,f,\lambda),
\]
where the functions
\[
J_P(\gamma,f,\lambda)=
\int_{G^\gamma(F_S)^1\bs G(F_S)^1}
e^{-\langle\lambda,H_Q^P(x)\rangle}
f(x^{-1}\gamma x)\,dx
\]
form a cofrugal $(G,Q)$-family because the weight factors do. In total, we obtain
\[
J_{G\prin,P}^T(f,\lambda)=
\epsilon_P\sum_{[\gamma]_S}
\tilde\zeta_P(S,\gamma,\lambda)J_P(\gamma,f,\lambda)
\hat\theta_Q^P(\lambda)^{-1}\theta_P^T(\lambda)^{-1}.
\]
By Lemmas \ref{Jprin} and~\ref{JasPI}, we have
\[
\sum_{[\gamma]_S}
\tilde\zeta_P(S,\gamma,\lambda)J_P(\gamma,f,\lambda)
=\tilde Z_{M\prin,M}^M(f_V^P,\lambda^P).
\]
We have seen in the proof of Lemma~\ref{singZ} that the zeta integrals with removed singularities on the right-hand side form a cofrugal $(G,Q)$-family, and we deduce the same property for the functions $\tilde\zeta_P(S,\gamma,\lambda)$ with fixed $S$ and~$\gamma$ by choosing $f$ supported in the $G(F_S)$-conjugacy class of~$\gamma$.

Summing the above formulas for $J_{G\prin,P}^T(f,\lambda)$ over~$P\supset Q$ and applying Lemma~\ref{prodGQ} with $d_P(\lambda)=e^{-\langle\lambda,T_P\rangle}J_P(\gamma,f,\lambda)$,
we obtain
\[
J_{G\prin}^T(f)=\epsilon_Q\sum_{[\gamma]_S}\sum_{P\supset Q}(\tilde\zeta_Q^P)'(S,\gamma,0)J_P^T(\gamma,f),
\]
where $(\tilde\zeta_Q^P)'(S,\gamma,\lambda)$ is as in Lemma~\ref{GQ} and
\[
J_P^T(\gamma,f)=\int_{G^\gamma(F_S)^1\bs G(F_S)^1}f(x^{-1}\gamma x)w_P^T(x)\,dx
\]
with the weight factor
\[
w_P^T(x)=\lim_{\lambda\to0}\sum_{P'\supset P}\epsilon_P^{P'}
e^{-\langle\lambda,H_P^{P'}(x)+T_{P'}\rangle}
\hat\theta_P^{P'}(\lambda)^{-1}\theta_{P'}(\lambda)^{-1}.
\]
Applying Lemma~\ref{prodGQ} again, we get
\[
w_P^T(x)=\sum_{P'\supset P}\epsilon_P^{P'}
v_P^{P'}(H_P^{P'}(x))v_{P'}(T_{P'})
\]
in terms of the relative versions of the function
\[
v_Q(X)=\int_{\af_Q^G}\Gamma'_Q(H,X)\,dH=\epsilon_Q\int_{\af_Q^G}\Gamma''_Q(H,X)\,dH,
\]
where the equality of the integrals follows from Lemma~\ref{G'}(iii). Since $v_Q(X)=\epsilon_Q v_Q(-X)$, we can also write
\[
w_P^T(x)=\sum_{P'\supset P}
v_P^{P'}(-H_P^{P'}(x))v_{P'}(T_{P'})
=v_P(T_P-H_P(x)),
\]
where the last equality follows with Lemma~\ref{prodGQ}.

There is an alternative formula. The $(M,Q\cap M)$-family giving rise to $(\tilde\zeta_Q^P)'(S,\gamma,\lambda)$ depends only on~$\gamma N(F_S)$, so we can write
\[
J_{G\prin}^T(f)=\epsilon_Q\sum_{P\supset Q}\sum_{[\gamma' N]_S}
(\tilde\zeta_Q^P)'(S,\gamma',0)J_P^T(\gamma' N,f),
\]
where $J_P^T(\gamma' N,f)$ is the sum of the $J_P^T(\gamma,f)$ over all $[\gamma]_S$ with $[\gamma N]_S=[\gamma'N]_S$. Recombining the integrals, we get
\[
J_P^T(\gamma'N,f)=
\epsilon_P\int_{L^{\gamma'N}(F_S)\bs L(F_S)}
f^P(l^{-1}\gamma' l)
v_P(H_P^T(l))\,dl.
\]

\section{Examples}

We are going to illustrate the constructions of this paper by some examples, restricting ourselves to subregular unipotent conjugacy clas\-ses in low-dimensional split classical groups.  Such a group $G$ is up to isogeny either the group $\GL(V)$, where $V$ is an $F$-vector space, or the subgroup stabilising a bilinear form $b$ or a symplectic form $\omega$ on~$V$. Parabolic subgroups are stabilisers of flags $V_0\subset\dots\subset V_r$ in~$V$, which have to be self-dual in the orthogonal and symplectic cases, i.~e., $V_i^\perp=V_{r-i}$ for each~$i$. To every conjugacy class of unipotent elements~$\gamma=\exp X$ or, equivalently, to every adjoint orbit of nilpotent elements~$X$, one associates a partition of the natural number~$\dim V$ (cf.\ section~5.1 of~\cite{CG}). We avoid the orthogonal case, in which both assignments are not quite bijective. Although in the notation $\{V_0,\dots,V_r\}$ for a flag one ought include $V_0=\{0\}$ and $V_r=V$, we will list only nonzero proper subspaces for brevity, so that $G$, considered as its own parabolic subgroup, appears as the stabiliser of the empty flag.

For each representative~$\gamma$, we will present the canonical flag determining the canonical parabolic~$Q$ of~$\gamma$, the corresponding prehomogeneous vector space defined in Theorem~\ref{canpar} and applied in section~\ref{redtoPV}, its basic relative invariants as in Theorem~\ref{zetaInt} and the split torus $A_{L_\nu}/A_G$, as mentioned after Hypothesis~\ref{convZeta}, by means of its faithful action on a subquotient of a suitable flag. We will also describe the poset $\mathcal P_\gamma\infl(F)$ defined in Theorem~\ref{induction2}(iv) and applied in Corollary~\ref{newJCcor}. If applicable, we will indicate the splitting of $C(F)$ into truncation classes defined in section~\ref{trclass} and the refined set~$\mathcal P_\gamma^{\textrm{min\,infl}}(F)$. For each parabolic $P=MN$ in this set, we will give the group $N^{[\gamma]}$ defined in Hypothesis~\ref{NC}, whose present notation was introduced in Lemma~\ref{newJC}.

\subsection{General linear group of rank~2}

Here $G(F)=\GL(V)$ with $\dim V=3$, and the subregular unipotent class  corresponds to the partition~$[2,1]$. The canonical flag of a representative $\gamma=\exp X$ is $\{V_-,V_+\}$, where
\[
V_-=\Im X,\qquad V_+=\Ker X,
\]
and $X$ defines an isomorphism $V/V_+\to V_-$. The Hasse diagram of these subspaces is shown in Figure~1.1. The corresponding prehomogeneous vector space is~$\Hom(V/V_+,V_-)$ with any nonzero linear function as basic relative invariant, and $A_{L_\nu}/A_G$ acts on $V_+/V_-$ by homotheties. The Hasse diagram of the parabolic subgroups $P$ in $\mathcal P_\gamma\infl(F)$, or rather their corresponding flags, is shown in Figure~1.2.
\begin{figure}[ht]\centering
\noindent\begin{minipage}[b]{0.2\textwidth}
\centering
$\xymatrix@R=1ex@C=0ex{
V\ar@{-}[d]\\
V_+\ar@{-}[d]\\
V_-\ar@{-}[d]\\
\{0\}
}$\medskip\par
Figure~1.1
\end{minipage}
\begin{minipage}[b]{0.4\textwidth}
\centering
$\xymatrix@=3ex{
&\emptyset\ar@{<-}[dl]\ar@{<-}[dr]&\\
\{V_-\}&&\{V_+\}
}$\vspace{4ex}\par
Figure~1.2
\end{minipage}
\end{figure}

For each such~$P$ with unipotent radical~$N$, the related group~$N'=N^{[\gamma]}$ is the unipotent radical of a parabolic~$P'$.  Here and below, we encode the assignment $P\mapsto P'$ in the Hasse diagram by an arrow between the corresponding flags. If no arrow starts at a flag, this means that we have $P'=P$ for the corresponding parabolic.

\subsection{Symplectic group of rank~2}
\label{Sp2}

Here $G(F)=\Sp(V,\omega)$ with $\dim V=4$, and the subregular unipotent class corresponds to the partition~$[2,2]$. The canonical flag of a representative $\gamma=\exp X$ is $\{V_0\}$, where
\[
V_0=\Ker X=\Im X.
\]
The element $X$ induces an isomorphism $V/V_0\to V_0$ and defines symmetric bilinear forms $b_+$ on $V/V_0$ and $b_-$ on~$V_0$ by
\[
b_+(u,v)=\omega(u,Xv)=b_-(Xu,Xv).
\]
If $b_+$ or, equivalently, $b_-$ splits over $F$ into a product of two linear forms, then there are isotropic lines $U_+/V_0$, $W_+/V_0$ for $b_+$ and $U_-$, $W_-$ for $b_-$. In this case $X$ determines four additional $F$-subspaces with the properties
\[
XU_+=U_+^\perp=U_-,\qquad XW_+=W_+^\perp=W_-.
\]
The Hasse diagram of these subspaces is shown in Figure~2.1 with the parts shaded that are only present in the split case. The corresponding prehomogeneous vector space is the space $\operatorname{Quad}(V_0)$ of quadratic forms on~$V_0$ with the discriminant as basic relative invariant. The torus $A_{L_\nu}/A_G$ acts as the split special orthogonal group on $V/V_0\cong V_0$ if $b_\pm$ is split, while it is trivial otherwise.
\begin{figure}[ht]\centering
\noindent\begin{minipage}[b]{0.25\textwidth}
\centering
$\xymatrix@R=1ex@C=0ex{
&V\ar@{-}[dd]&\\
\color{gray}U_+\ar@{.}[ur]\ar@{.}[dr]
&&\color{gray}W_+\ar@{.}[ul]\ar@{.}[dl]\\
&V_0\ar@{-}[dd]&\\
\color{gray}U_-\ar@{.}[ur]\ar@{.}[dr]
&&\color{gray}W_-\ar@{.}[ul]\ar@{.}[dl]\\
&\{0\}&
}$\medskip\par
Figure~2.1
\end{minipage}
\begin{minipage}[b]{0.15\textwidth}
\centering
$\xymatrix@R=4ex@C=1ex{
\emptyset\\
\{V_0\}\ar@{-}[u]
}$\vspace{5ex}\par
Figure~2.2
\end{minipage}
\begin{minipage}[b]{0.4\textwidth}
\centering
$\xymatrix@R=4ex@C=1ex{
&\emptyset&\\
\{U_-,U_+\}\ar@{->}[ur]&
&\{W_-,W_+\}\ar@{->}[ul]
}$\vspace{5ex}\par
Figure~$2.2'$
\end{minipage}
\end{figure}

The class $C(F)$ splits into two truncation classes $O$ and~$O'$ containing the elements for which the forms $b_\pm$ are anisotropic resp.~split. The Hasse diagram of $\mathcal P_\gamma^{\textrm{min\,infl}}(F)$ for $\gamma$ in $O$ resp.~$O'$ is shown in figures~2.2 resp.~$2.2'$ with the same encoding of the assignment $P\mapsto P'$ as above.

In this case, Lemma~\ref{JasPI} is true unconditionally, see \cite{HoWa} for details. The sum $Z_C^T(\varphi,\lambda)=Z_O^T(\varphi,\lambda)+Z_{O'}^T(\varphi,\lambda)$ of zeta integrals was called ``adjusted zeta function'' in~\cite{Yu}.

\subsection{General linear group of rank~3}

Here $G(F)=\GL(V)$ with $\dim V=4$, and the subregular unipotent class corresponds to the partition~$[3,1]$. The canonical flag of a representative $\gamma=\exp X$ is $\{V_-,V_+\}$, where
\begin{align*}
V_-&=\Ker X\cap\Im X=\Im X^2,\\
V_+&=\Ker X+\Im X=\Ker X^2.
\end{align*}
The corresponding prehomogeneous vector space is
\[
\Hom(V/V_+,V_+/V_-)\times\Hom(V_+/V_-,V_-),
\]
and the value of the basic relative invariant on $\nu=(\nu_1,\nu_2)$ in this space is the composition $\nu_2\circ\nu_1\in\Hom(V/V_+,V_-)$. The torus $A_{L_\nu}/A_G$ acts by homotheties on~$\Ker\nu_2$ stabilising~$\Im\nu_1$. Figure~2.1 shows the Hasse diagram of the pertinent subspaces together with $\Ker X$ and $\Im X$, whose stabilisers also belong to the set~$\mathcal P_\gamma\infl(F)$.
\begin{figure}[ht]\centering
\noindent\begin{minipage}[b]{0.25\textwidth}
\centering
$\xymatrix@R=1ex@C=0ex{
&V\ar@{-}[d]&\\
&V_+\ar@{-}[dl]\ar@{-}[dr]&\\
\Ker X\ar@{-}[dr]&&\Im X\ar@{-}[dl]\\
&V_-\ar@{-}[d]&\\
&0&
}$\medskip\par
Figure~3.1
\end{minipage}
\begin{minipage}[b]{0.55\textwidth}
\centering
$\xymatrix@!C=0.2em@!R=0.3em{
&&&\emptyset&&&\\
\{\Im X\}\ar[urrr]&&\{V_-\}\ar@{-}[ur]&
&\{V_+\}\ar@{-}[ul]&&\{\Ker X\}\ar[ulll]\\
&\{V_-,\Im X\}\ar@{-}[ul]\ar[ur]&
&\{V_-,V_+\}\ar@{-}[ul]\ar@{-}[ur]&
&\{\Ker X,V_+\}\ar[ul]\ar@{-}[ur]&
}$\vspace{3.2ex}\par
Figure~3.2
\end{minipage}
\end{figure}

The Hasse diagram of the latter poset appears in Figure~3.2 with the same encoding of the assignment $P\mapsto P'$ as above. There is no minimal parabolic contained in all its members, hence working with standard parabolic subgroups is inadequate. The stabilisers of $\Ker X$ and $\Im X$ are the first examples where the prehomogeneous affine space $\gamma N/N'$ is special under $P_{\gamma N}$, although the tangent prehomogeneous vector space $\mathfrak n/\mathfrak n'$ is not.
The zeta integral $Z_{C,G}^T(\varphi,\lambda)$ in this case cannot be handled yet. It is the first example in which the truncation function $\chi_\nu^T$ in Hypothesis~\ref{convZeta} really depends on~$\nu$.

\subsection{Symplectic group of rank~3}

Here $G(F)=\Sp(V,\omega)$ with $\dim V=6$, and the subregular unipotent class corresponds to the partition~$[4,2]$. The canonical flag of a representative $\gamma=\exp X$ is $\{V_-,V_0,V_+\}$, where
\begin{align*}
V_+&=\Ker X^3=\Ker X^2+\Im X,\\
V_0&=\Ker X^2\cap\Im X=\Ker X+\Im X^2,\\
V_-&=\Ker X\cap\Im X^2=\Im X^3.
\end{align*}
The element $X$ induces isomorphisms
\[
X:V_+/V_0\to V_0/V_-,\qquad X^2:V/V_+\to V_-
\]
and defines symmetric bilinear forms $b_+$ on $V_+/V_0$,\, $b_-$ on $V_0/V_-$ by
\[
b_+(u,v)=\omega(u,Xv)=b_-(Xu,Xv).
\]
The nonisotropic lines $\Im X/V_0$ for~$b_+$ and $\Ker X/V_-$ for~$b_-$ will also play a role, whence we have included $\Ker X$ and $\Im X$ in the Hasse diagram of subspaces shown in Figure~4.1. If $b_+$ or, equivalently, $b_-$ splits over $F$ into a product of two linear forms, then there are isotropic lines $U_+/V_0$, $W_+/V_0$ for~$b_+$ and
$U_-/V_-$, $W_-/V_-$ for~$b_-$. In this case $X$ determines four additional $F$-subspaces, which are shaded in the diagram, with the properties
\[
XU_+=U_+^\perp=U_-,\qquad XW_+=W_+^\perp=W_-.
\]
In any case, $\Hom(V/V_+,V_+/V_0)\times\operatorname{Quad}(V_+/V_0)$ is the associated prehomogeneous vector space. One basic relative invariant is the discriminant of the quadratic form, the other one is given by composition and takes values in $\operatorname{Quad}(V/V_+)$. The torus $A_{L_\nu}/A_G$ is trivial for all~$\nu$.
\begin{figure}[ht]\centering
\noindent\begin{minipage}[b]{0.25\textwidth}
\centering
$\xymatrix@R=0.9ex@C=0ex{
&V\ar@{-}[d]&\\
&V_+\ar@{-}[d]&\\
\color{gray}U_+\ar@{.}[ur]\ar@{.}[dr]&\Im X\ar@{-}[d]
   &\color{gray}W_+\ar@{.}[ul]\ar@{.}[dl]\\
&V_0\ar@{-}[d]&\\
\color{gray}U_-\ar@{.}[ur]\ar@{.}[dr]&\Ker X\ar@{-}[d]
   &\color{gray}W_-\ar@{.}[ul]\ar@{.}[dl]\\
&V_-\ar@{-}[d]&\\
&0&
}$\medskip\par
Figure~4.1
\end{minipage}
\begin{minipage}[b]{0.6\textwidth}
\centering
$\xymatrix@!C=1.4em{
&&&\emptyset\ar@{<-}[dlll]\ar@{-}[dl]\ar@{-}[dr]&&&\\
\{\Ker X,\Im X\}\ar@{-}[dr]
&&\{V_0\}\ar@{<-}[dl]\ar@{-}[dr]
&&\{V_-,V_+\}\ \ar@{-}[dl]\\                   
&\{\Ker X,V_0,\Im X\}&&\{V_-,V_0,V_+\}
}$\vspace{4.5ex}\par
Figure~4.2
\end{minipage}
\end{figure}

The class $C(F)$ splits into two truncation classes $O$ and~$O'$ containing the elements for which the forms $b_\pm$ are anisotropic resp.~split. The Hasse diagram of $\mathcal P_\gamma^{\textrm{min\,infl}}(F)$ for $\gamma$ in $O$ resp.~$O'$ is shown in Figures~4.2 resp.~$4.2'$.
\begin{figure}[ht]\centering
$\xymatrix@!C=2em{
&&\emptyset\ar@{-}[dll]\ar@{-}[d]\ar@{-}[drr]&&\\
\{U_-,U_+\}\ar@{-}[dr]&&\{V_-,V_+\}\ar@{-}[dl]\ar@{-}[dr]
&&\{W_-,W_+\}\ar@{-}[dl]\\
&\{V_-,U_-,U_+,V_+\}&&\{V_-,W_-,W_+,V_+\}&
}$\medskip\par
Figure~$4.2'$
\end{figure}

The class $O'$ is the first example of a truncation class for whose elements $\gamma$ the group $N'=N^{[\gamma]}$ cannot be chosen as the unipotent radical of a parabolic, hence cannot be encoded by arrows in the diagram. If $N$ is the unipotent radical of the stabiliser of~$(U_-,U_+)$, we may set
\[
\mathfrak n'=\{Z\in\mathfrak n\mid ZV\subset U_-,\, ZU_+=0\},
\]
whereas if $N$ is the stabiliser of~$(V_-,U_-,U_+,V_+)$, we may set
\[
\mathfrak n'=\{Z\in\mathfrak n\mid ZU_+\subset V_-,\,ZV_+\subset U_-\}
\]
and similarly with the letter $U$ replaced by~$W$. There are infinitely many~$N'$ for a fixed canonical parabolic, which suggests that one should search for another type of canonical subgroup attached to~$\gamma$.

\end{document}